\documentclass{amsart}

\usepackage{amsmath, amssymb}

\input xy
\xyoption{all}

\newtheorem{theorem}{Theorem}[section]
\newtheorem{lemma}[theorem]{Lemma}
\newtheorem{corollary}[theorem]{Corollary}
\newtheorem{fact}[theorem]{Fact}
\newtheorem{proposition}[theorem]{Proposition}

\newtheorem{question}[theorem]{Question}

\theoremstyle{definition}

\newtheorem{remark}[theorem]{Remark}
\newtheorem{definition}[theorem]{Definition}

\def\Frac{\operatorname{Frac}}
\def\alg{\operatorname{alg}}
\def\id{\operatorname{id}}
\def\span{\operatorname{span}}
\def\spec{\operatorname{Spec}}
\def\Div{\operatorname{div}}
\begin{document}

\title[Poisson algebras, model theory, differential-algebraic geometry]{Poisson algebras via model theory and differential-algebraic geometry}

\author{Jason Bell}
\address{Jason Bell\\
University of Waterloo\\
Department of Pure Mathematics\\
200 University Avenue West\\
Waterloo, Ontario \  N2L 3G1\\
Canada}
\email{jpbell@uwaterloo.ca}
\author{St\'ephane Launois}
\address{St\'ephane Launois\\
University of Kent\\
School of Mathematics, Statistics, and Actuarial Science \\
Canterbury, Kent \ CT2 7NZ \\
UK}
\email{s.launois@kent.ac.uk}

\author{Omar Le\'on S\'anchez}
\address{Omar Le\'on S\'anchez\\
McMaster University\\
Department of Mathematics and Statistics\\
1280 Main Street West\\
Hamilton, Ontario \  L8S 4L8\\
Canada}
\email{oleonsan@math.mcmaster.ca}

\author{Rahim Moosa}
\address{Rahim Moosa\\
University of Waterloo\\
Department of Pure Mathematics\\
200 University Avenue West\\
Waterloo, Ontario \  N2L 3G1\\
Canada}
\email{rmoosa@uwaterloo.ca}

\thanks{2010 {\em Mathematics Subject Classification}. Primary 17B63; Secondary 03C98, 12H05, 14L99.}

\date{\today}

\begin{abstract}
Brown and Gordon asked whether the Poisson Dixmier-Moeglin equivalence holds for any complex affine Poisson algebra; that is, whether the sets of Poisson rational ideals, Poisson primitive ideals, and Poisson locally closed ideals coincide.
In this article a complete answer is given to this question using techniques from differential-algebraic geometry and model theory.
In particular, it is shown that while the sets of Poisson rational and Poisson primitive ideals do coincide, in every Krull dimension at least four there are complex affine Poisson algebras with Poisson rational ideals that are not Poisson locally closed.
These counterexamples also give rise to counterexamples to the classical (noncommutative) Dixmier-Moeglin equivalence in finite $\operatorname{GK}$ dimension.
A weaker version of the Poisson Dixmier-Moeglin equivalence is proven for all complex affine Poisson algebras, from which it follows that the full equivalence holds in Krull dimension three or less.
Finally, it is shown that everything, except possibly that rationality implies primitivity, can be done over an arbitrary base field of characteristic zero.

\end{abstract}

\maketitle

\tableofcontents

\pagebreak
\section{Introduction}

\noindent
It is usually difficult to fully classify all the irreducible representations of a given algebra over a field.
As a substitute, one often focuses on the annihilators of the simple (left) modules, the so-called {\it primitive ideals}, which already provide a great deal of information on the representation theory of the algebra.
This idea was successfully developed by Dixmier \cite{Dixmier} and Moeglin \cite{Moeglin} in the case of enveloping algebras of finite-dimensional Lie algebras. In particular, their seminal work shows that primitive ideals can be characterised both topologically and algebraically among the prime ideals, as follows.
Let $A$ be a (possibly noncommutative) noetherian algebra over a field $k$.
If $P \in \mathrm{Spec}(A)$, then the quotient algebra $A/P$ is prime noetherian, and so by Goldie's Theorem (see for instance \cite[Theorem 2.3.6]{McConnellRobson}), we can localize $A/P$ at the set of all regular elements of $A/P$. The resulting algebra, denoted by $\mathrm{Frac}(A/P)$, is simple Artinian.
It follows from the Artin-Wedderburn Theorem that  $\mathrm{Frac}(A/P)$ is isomorphic to a matrix algebra over a division ring $D$. As a consequence, the centre of $\mathrm{Frac}(A/P)$ is isomorphic to the centre of the division ring $D$, and so this is a field extension of the base field $k$. A prime ideal $P \in \mathrm{Spec}(A)$ is \emph{rational} provided the centre of the Goldie quotient ring $\mathrm{Frac}(A/P)$ is algebraic over the base field $k$.
On the other hand, $P$ is said to be \emph{locally closed} if $\{P\}$ is a locally closed point of the prime spectrum $\mathrm{Spec}(A)$ of $A$ endowed with the Zariski topology (which still makes sense in the noncommutative world, see for instance \cite[4.6.14]{McConnellRobson}).
The results of Dixmier and Moeglin show that if $A$ is the enveloping algebra of a finite-dimensional complex Lie algebra, then the notions of primitive, locally closed, and rational coincide.
This result was later extended by Irving and Small to arbitrary base fields of characteristic zero~\cite{IrvingSmall}. 

The spectacular result of Dixmier and Moeglin has primarily led to research in three directions. First, it has been shown that under mild hypotheses, we have the following implications:
$$P \mbox{ locally closed } \Rightarrow \ P \mbox{ primitive } \Rightarrow \ P \mbox{ rational.}$$
Next, examples of algebras where the converse implications are not true were found. More precisely, Irving \cite{Irving} gave an example of a rational ideal which is not primitive, and Lorenz \cite{Lorenz} constructed an example of a primitive ideal which is not locally closed. Finally, the {\it Dixmier-Moeglin equivalence} (that is, the coincidence between the sets of primitive ideals, locally closed ideals and rational ideals) was established for important classes of algebras such as quantised coordinate rings  \cite{HodgesLevasseur1,HodgesLevasseur2,HodgesLevasseurToro,Joseph1,Joseph2,GoodearlLetzter}, twisted coordinate rings \cite{BRS} and Leavitt path algebras \cite{AbramsBellRangaswamy}. 

In the spirit of deformation quantization, the aim of this article is to study an analogue of the Dixmier-Moeglin equivalence for affine (i.e., finitely generated and integral) complex Poisson algebras.
Recall that a complex {\it Poisson algebra} is a commutative $\mathbb C$-algebra $A$ equipped with a Lie bracket $\{-,-\}$ (i.e. $\{ ~,~\}$ is bilinear, skew symmetric and satisfies Jacobi's identity) such that $\{-,x\}$ is a derivation for every $x\in A$, that is, for every $x,y,z\in A$, $\{yz,x\}=y\{z,x\}+z\{y,x\}$.  We point out that the derivations $\{-,x\}$ are trivial on $\mathbb{C}$.  It is natural to consider the Dixmier-Moeglin equivalence in this setting because most of the noncommutative algebras for which the equivalence has been established recently are noncommutative deformations of classical commutative objects. For instance, the quantised coordinate ring $O_q(G)$ of a semisimple complex algebraic group $G$ is a noncommutative deformation of the coordinate ring $O(G)$ of $G$. Moreover, the noncommutative product in $O_q(G)$ gives rise to a Poisson bracket on $O(G)$ via the well known semiclassical limit process (see for instance \cite[Chapter III.5]{BrownGoodearl}). As usual in (algebraic) deformation theory, it is natural to ask how the properties from one world translate into the other.
 
A {\it Poisson ideal} of a Poisson algebra $\big(A,\{-,-\}\big)$ is any ideal $I$ of $A$ such that $\{I,A\} \subseteq I$. A prime ideal which is also a  Poisson ideal 
is called a {\it Poisson prime} ideal. The set of Poisson prime ideals in $A$ forms the {\it
Poisson prime spectrum}, denoted $\mathrm{PSpec} A$, which is given the relative Zariski
topology inherited from $\mathrm{Spec} A$.
In particular, a Poisson prime ideal $P$ is {\em locally closed} if there is a nonzero $f\in A/P$ such that the localisation $(A/P)_f$ has no proper nontrivial Poisson ideals.

For any ideal $J$ of $A$, there is a largest Poisson ideal contained in $J$. This Poisson ideal is called the {\it Poisson core} of $J$. 
Poisson cores of the maximal ideals of $A$ are called {\it Poisson primitive ideals}.  The central role of the Poisson primitive ideals was pinpointed by Brown and Gordon. Indeed, they proved for instance that the defining ideals of the Zariski-closures of the symplectic leaves of a complex affine Poisson variety $V$ are precisely the Poisson-primitive ideals of the coordinate ring of $V$ \cite[Lemma 3.5]{BrownGordon}. The fact that the notion of Poisson primitive ideal is a Poisson analogue of the notion of primitive ideal is supported for instance by the following result due to Dixmier-Conze-Duflo-Rentschler-Mathieu and Borho-Gabriel-Rentschler-Mathieu, and expressed into (Poisson-)ideal-theoretic terms by Goodearl: Let $\mathfrak{g}$ be a solvable finite-dimensional complex Lie algebra. Then there is a homeomorphism between the Poisson primitive ideals of the symmetric algebra $S(\mathfrak{g})$ of $\mathfrak{g}$ (endowed with the Kirillov-Kostant-Souriau bracket) and the primitive ideals of the enveloping 
algebra $U(\mathfrak{g})$ of $\mathfrak{g}$ \cite[Theorem 8.11]{Goodearl2}.

The {\it Poisson center} of $A$ is the subalgebra
$$Z_p(A)= \{z\in A\mid \{z,-\} \equiv 0\}.$$
For any Poisson prime ideal $P$ of $A$, there is an induced Poisson bracket on
$A/P$, which extends uniquely to the quotient field $\mathrm{Frac}(A/P)$. We say that $P$ is {\it Poisson rational} if the
field $Z_p(\mathrm{Frac} (A/P))$ is algebraic over the base field $k$.

By analogy with the Dixmier-Moeglin equivalence for enveloping algebras, we say
that $A$ satisfies the {\it Poisson Dixmier-Moeglin equivalence} provided the following sets of Poisson prime ideals coincide:
\begin{enumerate}
\item The set of Poisson primitive ideals in $A$;
\item The set of Poisson locally closed ideals;
\item The set of Poisson rational ideals of $A$.
\end{enumerate}
If $A$ is an affine  Poisson algebra, then $(2)\subseteq (1)\subseteq (3)$ \cite[1.7(i), 1.10]{Oh}, so the main difficulty is whether $(3)\subseteq (2)$. 

The Poisson Dixmier-Moeglin equivalence has been established for Poisson algebras with suitable torus actions by Goodearl \cite{Goodearl}, so that many Poisson algebras arising as semiclassical limits of quantum algebras satisfy the Poisson Dixmier-Moeglin equivalence \cite{GoodearlLaunois}.
On the other hand, Brown and Gordon proved that the Poisson Dixmier-Moeglin equivalence holds for any affine complex Poisson algebra with only finitely many Poisson primitive ideals \cite[Lemma 3.4]{BrownGordon}.
Given these successes (and the then lack of counterexamples),  Brown and Gordon asked \cite[Question 3.2]{BrownGordon} whether the Poisson Dixmier-Moeglin equivalence holds for all affine complex Poisson algebras.
In this article, we give a complete answer to this question.
We show that $(3)=(1)$ but that $(3)\neq (2)$.
More precisely, we prove that in a finitely generated integral complex Poisson algebra any Poisson rational ideal is Poisson primitive (Theorem \ref{thm: ratprim}), but that for all $d \geq 4$ there exist finitely generated integral complex Poisson algebras of Krull dimension $d$ in which $(0)$ is Poisson rational but not Poisson locally closed (Corollary \ref{counterexample}).

We also prove that the hypothesis $d \geq 4$ is actually necessary to construct counterexamples; our Theorem \ref{lethree} says that the Poisson Dixmier-Moeglin equivalence holds in Krull dimension~$\leq 3$. 
This is deduced from a weak version of the Poisson Dixmier-Moeglin equivalence, where Poisson locally closed ideals are replaced by Poisson prime ideals $P$ such that the set $\mathcal{C}(P):=\{Q \in \mathrm{PSpec} A ~|~ Q\supset P, \mathrm{ht}(Q)=\mathrm{ht}(P)+1 \}$ is finite.
We prove that for any finitely generated integral complex Poisson algebra, a Poisson prime ideal is Poisson primitive if and only if $\mathcal{C}(P)$ is finite.
This is Theorem~\ref{thm: weak} below.

Finally, in Section 8 we show that most of our results extend to arbitrary fields of characteristic zero, and in Section 9 we observe that our results also provide new examples of algebras which do not satisfy the Dixmier-Moeglin equivalence.

What is particularly novel about the approach taken in this paper is that the counterexample comes from differential-algebraic geometry and the model theory of differential fields.
As is explained in Proposition~\ref{lem: pb} below, to a commutative $\mathbb C$-algebra $R$ equipped with a $\mathbb C$-linear derivation $\delta:R\to R$ we can associate a Poisson bracket on $R[t]$ many of whose properties can be read off from $(R,\delta)$.
In particular, $(0)$ will be a rational but not locally closed Poisson ideal of $R[t]$ if and only if the kernel of $\delta$ on $\Frac(R)$ is $\mathbb C$ and the intersection of all the nontrivial prime $\delta$-ideals (i.e., prime ideals preserved by $\delta$) is zero.
As we show in Section~\ref{diffalgsect}, the existence of such an $(R,\delta)$ can be deduced from the model theory of Manin kernels on abelian varieties over function fields, a topic that was at the heart of Hrushovski's model-theoretic proof of the function field Mordell-Lang conjecture~\cite{ML}.
We have written Section~\ref{diffalgsect} to be self-contained, translating as much as possible of the underlying model theory into statements of an algebro-geometric nature so that familiarity with model theory is not required.

Differential algebra is also related to the positive results obtained in this paper.
We prove in Section~\ref{sec: weak} that whenever $(A,\delta_1,\dots.\delta_m)$ is a finitely generated integral complex differential ring (with the derivations not necessarily commuting), if the intersection of the kernels of the derivations extended to the fraction field is $\mathbb C$ then there are only finitely many height one prime ideals preserved by all the derivations (Theorem~\ref{thm: main}).
This theorem, when $m=1$, is a special case of an old unpublished result of Hrushovski generalising a theorem of Jouanolou.
Our weak Dixmier-Moeglin equivalence is proved by applying the above theorem to a Poisson algebra $A$ with derivations given by $\delta_i=\{-,x_i\}$ where $\{x_1,\dots,x_m\}$ is a set of generators of $A$ over $\mathbb C$.

Throughout the remainder of this paper all algebras are assumed to be commutative.
Moreover, by an {\em affine} $k$-algebra we will mean a finitely generated $k$-algebra that is an integral domain.

\medskip
{\em Acknowledgements.}
We would like to thank Zo\'e Chatzidakis, Martin Bays, James Freitag, Dave Marker, Colin Ingalls, Ronnie Nagloo, and Michael Singer for conversations that were useful in the writing of this paper.
The last two authors would also like to thank MSRI for its generous hospitality during the very stimulating ``Model Theory, Arithmetic Geometry and Number Theory" programme of spring~2014, where part of this work was done.

\bigskip
\section{The differential structure on a Poisson algebra}
\label{sect-diffpoi}
\noindent
In this section we discuss the differential structure induced by a Poisson bracket.

Suppose $\big(A,\{-,-\}\big)$ is an affine Poisson $k$-algebra for some field $k$, and $\{x_1,\dots,x_m\}$ is a fixed set of generators.
Let $\delta_1,\dots,\delta_m$ be the operators on $A$ defined by
$$\delta_i(a):=\{a,x_i\}$$
for all $a\in A$.
Then each $\delta_i$ is a $k$-linear derivation on $A$.
Note, however, that these derivations need not commute.
Nevertheless, we consider $(A,\delta_1,\dots,\delta_m)$ as a differential ring, and can speak about differential ideals (i.e., ideals preserved by each of $\delta_1,\dots,\delta_m$), the subring of differential constants (i.e., elements $a\in A$ such that $\delta_i(a)=0$ for $i=1,\dots,m$), and so on.
There is a strong connection between the Poisson and differential structures on $A$.
For example, one checks easily, using that $\{a,-\}$ is also a derivation on $A$, that
\begin{itemize}
\item[(i)]
{\em an ideal of $A$ is a Poisson ideal if and only if it is a differential ideal.}
\end{itemize}
From this we get, more or less immediately, the following differential characterisations of when a prime Poisson ideal $P$ of $A$ is locally closed and primitive:
\begin{itemize}
\item[(ii)]
{\em $P$ is locally closed if and only if the intersection of all the non-zero prime differential ideals in $A/P$ is not trivial.}
\item[(iii)]
{\em $P$ is primitive if and only if there is a maximal ideal in $A/P$ that does not contain any nontrivial differential ideals.}
\end{itemize}
To characterise rationality we should extend to $F=\Frac(A)$, in the canonical way, both the Poisson and differential structures on $A$.
It is then not difficult to see that the Poisson center of $\big(F,\{-,-\}\big)$ is precisely the subfield of differential constants in $(F,\delta_1,\dots,\delta_m)$.
In particular,
\begin{itemize}
\item[(iv)]
{\em a prime Poisson ideal in $A$ is rational if and only if the field of differential constants of $\Frac(A/P)$ is algebraic over $k$.}
\end{itemize}

\begin{remark}
The above characterisations are already rather suggestive to those familiar with the model theory of differentially closed fields; for example,
locally closed corresponds to the generic type of a differential variety being isolated, and rationality corresponds to that generic type being weakly orthogonal to the constants.
Note however that the context here is several derivations that may not commute.
In order to realise the model-theoretic intuition, therefore, something must be done.
One possibility is to work with the model theory of partial differential fields where the derivations need not commute.
Such a theory exists and is tame; for example, it is an instance of the formalism worked out in~\cite{moosa-scanlon}.
On the other hand, in this case one can use a trick of Cassidy and Kolchin (pointed out to us by Michael Singer) to pass to a commuting context after replacing the derivations by certain $F$-linear combinations of them.
Indeed, if $p_{i,j}\in \mathbb C[t_1,\ldots,t_m]$ is such that $\{x_j,x_i\}=p_{i,j}(x_1,\ldots,x_m)$ for all $i,j=1,\ldots,m$, then an easy computation (using the Jacobi identity) yields
$\displaystyle [\delta_i,\delta_j]=\sum_{k=1}^m \frac{\partial p_{i,j}}{\partial t_k}(x_1,\ldots,x_m)\delta_k.$
Thus, the $F$-linear span of the derivations $\{\delta_1,\ldots,\delta_m\}$ has the additional structure of a Lie ring.
It follows by Lemma 2.2 of~\cite{Singer} that this space of derivations has an $F$-basis consisting of \emph{commuting} derivations (see also~\cite[Chapter~9, \S5, Proposition~6]{KolchinDAG}), and one could work instead with those derivations.
But in fact we do not pursue either of these directions.
Instead, for the positive results of this paper we give algebraic proofs of whatever is needed about rings with  possibly non-commuting derivations and avoid any explicit use of model theory whatsoever.
For the negative results we associate to an ordinary differential ring a Poisson algebra of one higher Krull dimension (see Proposition~\ref{lem: pb}), and then use the model theory of ordinary differentially closed fields to build counterexamples in Poisson algebra.
\end{remark}

The following well known prime decomposition theorem for Poisson ideals can be seen as an illustration of how the differential structure on a Poisson algebra can be useful.

\begin{lemma}
\label{lem: decomp} Let $k$ be a field of characteristic zero. 
If $I$ is a Poisson ideal in an affine Poisson $k$-algebra $A$, then the radical of $I$ and all the minimal prime ideals over $I$ are Poisson.
\end{lemma}

\begin{proof}
Because of~(i) it suffices to prove the lemma with ``differential'' in place of ``Poisson''.  This result can be found in Dixmier \cite[Lemma 3.3.3]{Dixmierbook}.
\end{proof}

\bigskip
\section{Rational implies Primitive}
\noindent
In order to prove that Poisson rational implies Poisson primitive in affine complex Poisson algebras, we will make use of the differential-algebraic fact expressed in the following lemma.
This is our primary method for producing new constants in differential rings, and will be used again in Section~\ref{sec:jouanolou}.

\begin{lemma}
\label{lem: diffcenter}
 Let $k$ be a field and $A$ an integral $k$-algebra equipped with $k$-linear derivations $\delta_1,\dots,\delta_m$.
Suppose that there is a finite-dimensional $k$-vector subspace $V$ of $A$ and a set $\mathcal{S}$ of ideals satisfying:
\begin{itemize}
\item[(i)]
$\delta_i(I)\subseteq I$ for all $i=1,\dots,m$ and $I\in\mathcal S$,
\item[(ii)]
$\bigcap\mathcal{S} = (0)$, and
\item[(iii)]
 $V\cap I\neq (0)$ for all $I\in \mathcal{S}$.
\end{itemize}
Then there exists $f\in\Frac(A)\setminus k$ with $\delta_i(f)=0$ for all $i=1,\dots,m$.
\end{lemma}

\begin{proof}
We proceed by induction on $d=\dim V$.
The case of $\dim V=1$ is vacuous as then assumptions~(ii) and~(iii) are inconsistent.
Suppose that $d>1$ and fix a basis $\{v_1,\dots,v_d\}$ of $V$.
Let $\delta\in\{\delta_1,\dots,\delta_m\}$ be such that not all of $\delta\big(\frac{v_1}{v_d}\big),\dots,\delta\big(\frac{v_{d-1}}{v_d}\big)$ are zero.
If this were not possible, then each $\frac{v_j}{v_d}$ would witness the truth of the lemma and we would be done.
Letting
$u_j:=v_d^2\delta\big(\frac{v_j}{v_d}\big)$ for $j=1,\dots,d-1$,
we have that not all $u_1,\dots,u_{d-1}$ are zero.
It follows that
$$L:=\{(c_1,\ldots ,c_{d-1})\in k^{d-1}~:~\sum_{j=1}^{d-1} c_j  u_j \ = \ 0\}$$
is a proper subspace of $k^{d-1}$, and hence
$$W:=kv_d+\left\{\sum_{j=1}^{d-1} c_j v_j~:~(c_1,\ldots ,c_{d-1})\in L\right\}$$
is a proper subspace of $V$.

We prove the lemma by applying the induction hypothesis to $W$ with
$$\mathcal T:=\{I\in\mathcal S:I\cap \span_{k}\big(u_1,\dots,u_{d-1}\big)=(0)\}.$$
We only need to verify the assumptions.
Assumption~(i) holds {\em a fortiori} of $\mathcal T$.

Toward assumption~(ii), note that if $\bigcap(\mathcal S\setminus\mathcal T)=(0)$, then we are done by the induction hypothesis applied to $\span_{k}\big(u_1,\dots,u_{d-1}\big)$ with $\mathcal S\setminus\mathcal T$.
Hence we may assume that  $\bigcap(\mathcal S\setminus\mathcal T)\neq (0)$, and so as $A$ is an integral domain $\bigcap\mathcal T=(0)$.

It remains to check assumption~(iii): we claim that for each $I\in\mathcal T$, $W\cap I\neq(0)$.
Indeed, since $I\cap V\neq (0)$, we have in $I$ a nonzero element of the form $\displaystyle v:=\sum_{j=1}^{d} c_jv_j$.
As $I$ is preserved by $\delta$ we get
\begin{eqnarray*}
\delta(v)v_d - v\delta(v_d) &=& v_d^2\delta(\frac{v}{v_d})\\
&=&\sum_{j=1}^{d-1} c_ju_j
\end{eqnarray*}
is also in $I$.
But as $I\cap\span_{k}\big(u_1,\dots,u_{d-1}\big)=(0)$ by choice of $\mathcal T$, we must have $\displaystyle \sum_{j=1}^{d-1} c_ju_j=0$.
Hence $(c_1,\dots,c_{d-1})\in L$ and $v\in W$ by definition.
\end{proof}

We now prove that rational implies primitive.
Note that the converse is well known \cite[1.7(i), 1.10]{Oh}.

\begin{theorem} 
\label{rationalimpliesprimitive}
Let $A$ be a complex affine Poisson algebra and $P$ a Poisson prime ideal of $A$.
If $P$ is Poisson rational then $P$ is Poisson primitive.  
\label{thm: ratprim}
\end{theorem}

\begin{proof}
By replacing $A$ by $A/P$ if necessary, we may assume that $P=(0)$. 
Let $\mathcal{S}$ denote the set of nonzero Poisson prime ideals $Q$ such that $Q$ does not properly contain a nonzero Poisson prime ideal of $A$.  We claim that $\mathcal{S}$ is countable.  To see this, let $V$ be a finite-dimensional subspace of $A$ that contains $1$ and contains a set of generators for $A$.  We then let $V^n$ denote the span of all products of elements of $V$ of length at most $n$.  By assumption, we have
$$A \ = \ \bigcup_{n\ge 0} V^n$$ and in particular every nonzero ideal of $A$ intersects $V^n$ nontrivially for $n$ sufficiently large.

We claim first that $$\mathcal{S}_n:=\{ Q\in \mathcal{S}\colon Q\cap V^n\neq (0)\}$$ 
has nontrivial intersection.
Toward a contradiction, suppose $\bigcap\mathcal S_n=(0)$.
Fixing generators $\{x_1,\dots,x_m\}$ of $A$ over $\mathbb C$, let $\delta_i$ be the derivation given by $\{-,x_i\}$ for $i=1,\dots,m$.
Since the ideals in $\mathcal S_n$ are Poisson, they are differential.
The assumptions of Lemma~\ref{lem: diffcenter} are thus satisfied, and we have $f\in\Frac(A)\setminus\mathbb C$ with $\delta_i(f)=0$ for all $i=1,\dots,m$.
This contradicts the Poisson rationality of $(0)$ in $A$, see statement~(iv) of Section~\ref{sect-diffpoi}.
Hence
$\displaystyle L_n\ := \ \bigcap_{Q\in \mathcal{S}_n} Q$ is nonzero.

Next we claim that $\mathcal S_n$ is finite.
Since $A$ is a finitely generated integral domain and $L_n$ is a nonzero radical Poisson ideal, we have by Lemma~\ref{lem: decomp} that in the prime decomposition $L_n = P_1\cap \cdots \cap P_m$ each $P_i$ is a nonzero prime Poisson ideal.
As each ideal in $\mathcal{S}_n$ is prime and contains $L_n$, and hence also some $P_i$, 
it follows by choice of $\mathcal S$ that $\mathcal{S}_n\subseteq \{P_1,\ldots ,P_m\}$.

So $\displaystyle \mathcal{S}=\bigcup_{n\geq 0}\mathcal S_n$ is countable.   We let $Q_1,Q_2,\ldots $ be an enumeration of the elements of $\mathcal{S}$. 
For each $i$, there is some nonzero $f_i\in Q_i$.
We let $T$ denote the countable multiplicatively closed set generated by the $f_i$.  Then $B:=T^{-1}A$ is a countably generated complex algebra.  It follows that $B$ satisfies the Nullstellensatz \cite[II.7.16]{BrownGoodearl} and since $\mathbb{C}$ is algebraically closed, we then have that $B/I$ is $\mathbb{C}$ for every maximal ideal $I$ of $B$.
If $I$ is a maximal ideal of $B$ and $J:=I\cap A$ then $A/J$ embeds in $B/I$, and hence $A/J\cong \mathbb{C}$ and so $J$ is a maximal ideal of $A$.  By construction, $J$ does not contain any ideal in $\mathcal S$, and so $(0)$ is the largest Poisson ideal contained in $J$.
That is, $(0)$ is Poisson primitive, as desired.  
\end{proof}

We note that this proof only requires the uncountability of $\mathbb{C}$; it works over any uncountable base field $k$.  If one follows the proof, we cannot in general ensure that $B/I\cong k$, but we have that it is an algebraic extension of $k$ since $B$ still satisfies the Nullstellensatz.
We then have that $A/J$ embeds in an algebraic extension of $k$ and thus it too is an algebraic extension of $k$ and we obtain the desired result.

\section{A differential-algebraic example}
\label{diffalgsect}
\noindent
Our goal in this section is to prove the following theorem.

\begin{theorem}
\label{diffex}
There exists a complex affine algebra $R$ equipped with a derivation $\delta$ such that
\begin{itemize}
\item[(i)]
the field of constants of $\big(\Frac(R),\delta\big)$ is $\mathbb C$, and
\item[(ii)]
the intersection of all nontrivial prime differential ideals of $R$ is zero.
\end{itemize}
In fact, such an example can be found of any Krull dimension $\geq 3$.
\end{theorem}

To the reader sufficiently familiar with the model theory of differentially closed fields, this theorem should not be very surprising: the $\delta$-ring $R$ that we will produce will be the coordinate ring of a $D$-variety that is related to the Manin kernel of a simple nonisotrivial abelian variety defined over a function field over $\mathbb C$.
We will attempt, however, to be as self-contained and concrete in our construction as possible.
We will at times be forced to rely on results from model theory for which we will give references from the literature.
We begin with some preliminaries on differential-algebraic geometry.

\medskip
\subsection{Prolongations, $D$-varieties, and finitely generated $\delta$-algebras}
\label{dvarieties}
Let us fix a differential field $(k,\delta)$ of characteristic zero.
Suppose that $V\subseteq\mathbb A^n$ is an irreducible affine algebraic variety over $k$.
Then by the {\em prolongation} of $V$ is meant the algebraic variety $\tau V\subseteq \mathbb A^{2n}$ over $k$ whose defining equations are
\begin{eqnarray*}
P(X_1,\dots,X_n)&=&0\\
P^\delta(X_1,\dots,X_n)+\sum_{i=1}^n\frac{\partial P}{\partial X_i}(X_1,\dots,X_n)\cdot Y_i&=&0
\end{eqnarray*}
for each $P\in I(V)\subset k[X_1,\dots,X_n]$.
Here $P^\delta$ denotes the polynomial obtained by applying $\delta$ to all the coefficients of $P$.
The projection onto the first $n$ coordinates gives us a surjective morphism $\pi:\tau V\to V$.
Note that if $a\in V(K)$ is any point of $V$ in any differential field extension $(K,\delta)$ of $(k,\delta)$, then $\nabla(a):=(a,\delta a)\in\tau V(K)$.

If $\delta$ is trivial on $k$ then $\tau V$ is nothing other than $TV$, the usual tangent bundle of $V$.
In fact, this is the case as long as $V$ is defined over the constant field of $(k,\delta)$ because in the defining equations for $\tau V$ given above we could have restricted ourselves to polynomials $P$ coming from $I(V)\cap F[X_1,\dots,X_n]$ for any field of definition $F$ over $V$.
In general, $\tau V$ will be a torsor for the tangent bundle; for each $a\in V$ the fibre $\tau_a V$ is an affine translate of the tangent space $T_aV$.

Taking prolongations is a functor which acts on morphisms $f:V\to W$ by acting on their graphs.
That is, $\tau f:\tau V\to \tau W$ is the morphism whose graph is the prolongation of the graph of $f$, under the canonical identification of $\tau(V\times W)$ with $\tau V\times\tau W$.

We have restricted our attention here to the affine case merely for concreteness.
The prolongation construction extends to abstract varieties by patching over an affine cover in a natural and canonical way.
Details can be found in $\S1.9$ of~\cite{marker}.

This following formalism was introduced by Buium~\cite{buium-dagfd} as an algebro-geometric approach to Kolchin's differential algebraic varieties.

\begin{definition}
\label{defn-dvar}
A {\em $D$-variety over $k$} is a pair $(V,s)$ where $V$ is an irreducible algebraic variety over $k$ and $s:V\to\tau V$ is a regular section to the prolongation  defined over $k$.
A {\em $D$-subvariety} of $(V,s)$ is then a $D$-variety $(W,t)$ where $W$ is a closed subvariety of $V$ and $t=s|_W$.
\end{definition}

An example of a $D$-variety is any algebraic variety $V$ over the constant field of $(k,\delta)$ and equipped with the zero section to its tangent bundle.
Such $D$-varieties, and those isomorphic to them, are called {\em isotrivial}.
(By a morphism of $D$-varieties $(V,s)$ and $(W,t)$ we mean a morphism $f:V\to W$ such that $t\circ f=\tau f\circ s$).
We will eventually construct $D$-varieties, both over the constants and not, that are far from isotrivial.

Suppose now that $V\subseteq\mathbb A^n$ is an affine variety over $k$ and $k[V]$ is its co-ordinate ring.
Then the possible affine $D$-variety structures on $V$ correspond bijectively to the extensions of $\delta$ to a derivation on $k[V]$.
Indeed, given $s:V\to\tau V$, write $s(X)=\big(X,s_1(X),\dots,s_n(X)\big)$ in variables $X=(X_1,\dots,X_n)$.
There is a unique derivation on the polynomial ring $k[X]$ that extends $\delta$ and takes $X_i\to s_i(X)$.
The fact that $s$ maps $V$ to $\tau V$ will imply that this induces a derivation on $k[V]=k[X]/I(V)$.
Conversely suppose we have an extension of $\delta$ to a derivation on $k[V]$, which we will also denote by $\delta$.
Then we can write $\delta\big(X_i+I(V)\big)=s_i(X)+I(V)$ for some polynomials $s_1,\dots,s_n\in k[X]$.
The fact that $\delta$ is a derivation on $k[V]$ extending that on $k$ will imply that $s=(\id,s_1,\dots,s_n)$ is a regular section to $\pi:\tau V\to V$.
It is not hard to verify that these correspondences are inverses of each other. Moreover, the usual correspondence between subvarieties of $V$ defined over $k$ and prime ideals of $k[V]$, restricts to a correspondence between the $D$-subvarieties of $(V,s)$ defined over $k$ and the prime differential ideals of $k[V]$.

From now on, whenever we have an affine $D$-variety $(V,s)$ over $k$ we will denote by $\delta$ the induced derivation on $k[V]$ described above.
In fact, we will also use $\delta$ for its unique extension to the fraction field $k(V)$.

\subsection{The Kolchin topology and differentially closed fields}
While the algebro-geometric preliminaries discussed in the previous section are essentially sufficient for explaining the construction of the example whose existence Theorem~\ref{diffex} asserts, the proof that this construction is possible, and that it does the job, will use some model theory of differentially closed fields.
We therefore say a few words on this now, referring the reader to Chapter~2 of~\cite{mof} for a much more detailed introduction to the subject.

Given any differential field of characteristic zero, $(k,\delta)$, for each $n>0$, the derivation induces on $\mathbb A^n(k)$ a noetherian topology that is finer than the Zariski topology, called the {\em Kolchin topology}.
Its closed sets are the zero sets of {\em $\delta$-polynomials}, that is, expressions of the form $P(X,\delta X,\delta^2X,\dots,\delta^\ell X)$ where $\delta^iX=(\delta^iX_1,\dots,\delta^iX_n)$ and $P$ is an ordinary polynomial over $k$ in $(\ell+1)n$ variables.

Actually the Kolchin topology makes sense on $V(k)$ for any (not necessarily affine) algebraic variety $V$, by considering the Kolchin topology on an affine cover.
One can then develop $\delta$-algebraic geometry in general, for example the notions of {\em $\delta$-regular} and {\em $\delta$-rational} maps between Kolchin closed sets, in analogy with classical algebraic geometry.

The Kolchin closed sets we will mostly come across will be of the following form.  Suppose that $(V,s)$ is a $D$-variety over $k$. Then set
$$(V,s)^\sharp(k):=\{a\in V(k):s(a)=\nabla(a)\}.$$
Recall that $\nabla:V(k)\to\tau V(k)$ is the map given by $a\mapsto (a,\delta a)$.
So to say that $s(a)=\nabla(a)$ is to say, writing $s=(\id,s_1,\dots,s_n)$ in an affine chart, that $\delta a_i=s_i(a)$ for all $i=1,\dots,n$.
As the $s_i$ are polynomials, $(V,s)^\sharp$ is Kolchin closed; in fact it is defined by order $1$ algebraic differential equations.
While these Kolchin closed sets play a central role, not every Kolchin closed set we will come across will be of this form.

Just as the geometry of Zariski closed sets is only made manifest when the ambient field is algebraically closed, the appropriate universal domain for the Kolchin topology is a {\em differentially closed} field $(K,\delta)$ extending $(k,\delta)$.
This means that any finite system of $\delta$-polynomial equations and inequations over $K$ that has a solution in some differential field extension of $(K,\delta)$, already has a solution in $(K,\delta)$.
In particular, $K$ is algebraically closed, as is its field of constants.
One use of differential-closedness is the following property, which is an instance of the ``geometric axiom'' for differentially closed fields (statement~(ii) of Section~2 of~\cite{pierce-pillay}).

\begin{fact}
\label{densesharp}
Suppose $(V,s)$ is a $D$-variety over $k$.
Let $(K,\delta)$ be a differentially closed field extending $(k,\delta)$.
Then $(V,s)^\sharp(K)$ is Zariski dense in $V(K)$.
In particular, an irreducible subvariety $W\subseteq V$ over $k$ is a $D$-subvariety if and only if $W\cap(V,s)^\sharp(K)$ is Zariski dense in $W(K)$.
\end{fact}

\subsection{A $D$-variety construction over function fields}
\label{aadgroup}
We aim to prove Theorem~\ref{diffex} by constructing a $D$-variety over $\mathbb C$ whose co-ordinate ring will have the desired differential-algebraic properties.
But we begin with a well known construction of a $D$-structure on the universal vectorial extension of an abelian variety.
This is part of the theory of the Manin kernel and was used by both Buium and Hrushovski in their proofs of the function field Mordell-Lang conjecture.
There are several expositions of this material available, our presentation is informed by Marker~\cite{marker} and Bertrand-Pillay~\cite{bertrandpillay}.

Fix a differential field $(k,\delta)$ whose field of constants is $\mathbb C$ but $k\neq\mathbb C$.
(The latter is required because we will eventually need an abelian variety over $k$ that is not isomorphic to any defined over $\mathbb C$.)
In practice $k$ is taken to be a function field over $\mathbb C$.
For example, one can consider $k=\mathbb C(t)$ and $\delta=\frac{\text d}{\text{dt}}$.

Let $A$ be an abelian variety over $k$, and let $\widehat A$ be the {\em universal vectorial extension of $A$}.
So $\hat A$ is a connected commutative algebraic group over $k$ equipped with a surjective morphism of algebraic groups $p:\widehat A\to A$ whose kernel is isomorphic to an algebraic vector group, and moreover, we have the universal property that $p$ factors uniquely through every such extension of $A$ by a vector group.
The existence of this universal object goes back to Rosenlicht~\cite{rosenlicht}, but see also the more modern and general algebro-geometric treatment in~\cite{mazurmessing}.
The dimension of $\widehat A$ is twice that of $A$.

The prolongation $\tau\widehat A$ inherits the structure of a connected commutative algebraic group in such a way that $\pi:\tau\widehat A\to \widehat A$ is a morphism of algebraic groups.
This is part of the functoriality of prolongations, see $\S2$ of~\cite{marker} for details on this induced group structure.
The kernel of $\pi$ is the vector group $\tau_0\widehat A$ which is isomorphic to the Lie algebra $T_0\widehat A$.
In fact, since $\tau\widehat A$ is a commutative algebraic group, one can show that $\tau\widehat A$ is isomorphic to the direct product $\widehat A\times\tau_0\widehat A$.

We can now put a $D$-variety structure on $\widehat A$.
Indeed it will be a {\em $D$-group} structure, that is, the regular section $s:\widehat A\to\tau\widehat A$ will be also a group homomorphism.
We obtain $s$ by the universal property that $\widehat A$ enjoys:
the composition $p\circ \pi:\tau\widehat A\to A$ is again an extension of $A$ by a vector group and so there is a unique morphism of algebraic groups $s:\widehat A\to\tau\widehat A$ over $k$ such that $p=p\circ\pi\circ s$.
It follows that $s$ is a section to $\pi$ and so $(\widehat A,s)$ is a $D$-group over $k$.

But $(\widehat A,s)$ is not yet the $D$-variety we need to prove Theorem~\ref{diffex}.
Rather we will need a certain canonical quotient of it.

\begin{lemma}
$(\widehat A,s)$ has a unique maximal $D$-subgroup $(G,s)$ over $k$ that is contained in $\ker(p)$.
\end{lemma}

\begin{proof}
This is from the model theory of differentially closed fields.
Given any $D$-group $(H,s)$ over $k$, work in a differentially closed field $K$ extending $k$.
From Fact~\ref{densesharp} one can deduce that a connected algebraic subgroup $H'\leq H$ over $k$ is a $D$-subgroup if and only if $H'\cap (H,s)^\sharp(K)$ is Zariski dense in $H'$.
So, in our case, letting $G$ be the Zariski closure of $\ker(p)\cap(\widehat A,s)^\sharp(K)$ establishes the lemma.
\end{proof}

Let $V$ be the connected algebraic group $\widehat A/G$.
Then $V$ inherits the structure of a $D$-group which we denote by $\overline s:V\to\tau V$.
In fact, $\tau V$ is canonically isomorphic to $\tau\widehat A/\tau G$ and $\overline s(a+G)=s(a)+\tau G$.
This $D$-group $(V,\overline s)$ over $k$ is the one we are interested in.
Note that $p:\widehat A\to A$ factors through an algebraic group morphism $V\to A$, and so in particular $\dim A\leq\dim V\leq 2\dim A$.

\begin{remark}
\label{nadgroup}
It is known that $G=\ker(p)$, and so $V=A$, if and only if $A$ admits a $D$-group structure if and only if $A$ is isomorphic to an abelian variety over $\mathbb C$, see $\S3$ of~\cite{bertrandpillay}.
So in the case when $A$ is an elliptic curve that is not defined over $\mathbb C$, it follows that $\dim V=2$.
\end{remark}

The following well known fact reflects important properties of the Manin kernel that can be found, for example, in~\cite{marker}.
We give some details for the reader's convenience, at least illustrating what is involved, though at times simply quoting results appearing in the literature.

\begin{fact}
\label{dvarexample}
Let $(V,\overline s)$ be the $D$-variety constructed above.
Then
\begin{itemize}
\item[(i)]
$(V,\overline s)^\sharp(k^{\alg})$ is Zariski-dense in $V(k^{\alg})$.
\item[(ii)]
Suppose in addition that $A$ has no proper infinite algebraic subgroups (so is a {\em simple} abelian variety) and is not isomorphic to any abelian variety defined over $\mathbb C$.
Then the field of constants of 
 $\big(k(V),\delta\big)$ is $\mathbb C$.
\end{itemize}
\end{fact}

\begin{remark}
As $(V,\overline s)$ is not affine, we should explain what differential structure we are putting on $k(V)$ in part~(ii).
Choose any affine open subset $U\subset V$, then $\tau U$ is affine open in $\tau V$, and $\overline s$ restricts to a $D$-variety structure on $U$.
We thus obtain, as explained in $\S$\ref{dvarieties}, an extension of $\delta$ to $k[U]$, and hence to $k(U)=k(V)$.
This construction does not depend on the choice of affine open $U$ (since $V$ is irreducible).
\end{remark}

\begin{proof}[Sketch of proof of Fact~\ref{dvarexample}]
Part~(i).
The group structure on $\tau V$ is such that $\nabla:V(k^{\alg})\to \tau V(k^{\alg})$ is a group homomorphism.
Hence the difference $\overline s-\nabla:V(k^{\alg})\to\tau V(k^{\alg})$ is a group homomorphism.
Its image lies in $\tau_0 V$ which is isomorphic to the vector group $T_0 V$.
Hence all the torsion points of $V(k^{\alg})$ must be in the kernel of $\overline s-\nabla$, which is precisely $(V,s)^\sharp(k^{\alg})$.
So it suffices to show that the torsion of $V(k^{\alg})$ is Zariski dense in $V(k^{\alg})$.
Now the torsion in $A(k^{\alg})$ is Zariski dense as $A(k^{\alg})$ is an abelian variety over~$k$.
Moreover, since $\ker(p)$ is divisible (it is a vector group), every torsion point of $A(k^{\alg})$ lifts to a torsion point of $\widehat A(k^{\alg})$.
One of the properties of the universal vectorial extension is that no proper algebraic subgroup of $\widehat A$ can project onto $A$ (this is~$4.4$ of~\cite{marker}) .
So the torsion of $\widehat A(k^{\alg})$ must be Zariski dense in $\widehat A(k^{\alg})$.
But $V=\widehat A/G$, and so the torsion of $V(k^{\alg})$ is also Zariski dense in $V(k^{\alg})$, as desired.

Part~(ii).
This part uses quite a bit more model theory than we have introduced so far, and as it is a known result, we content ourselves here with attempting only to give to the non model theorist some idea why the existence of a new differential constant in $\big(k(V),\delta)$ is inconsistent with $A$ not being defined over $\mathbb C$.

Work over a sufficiently large differentially closed field $K$ extending $k(V)$ and with field of constants $C$.
Then $C\cap k=\mathbb C$, so it suffices to show that a new differential constant in $\big(k(V),\delta)$ implies that $A$ is defined over $C$.

We will use model theoretic properties of the {\em Manin kernel} of~$A$;
let $A^\sharp\leq A(K)$ denote the Kolchin closure of the torsion subgroup of $A$.
It is a Zariski dense Kolchin closed subgroup of $A(K)$.
Note that, despite the notation, the Manin kernel is not itself the ``sharp'' points of a $D$-variety.
However, Proposition~3.9 of~\cite{bertrandpillay} tells us that $V\to A$ restricts to a $\delta$-rational isomorphism $(V,\overline s)^\sharp(K)\to A^\sharp$.

Suppose toward a contradiction that there is $f\in k(V)\setminus k$ with $\delta(f)=0$.
So $f\in C$, which means that as a rational function on $V$, $f$ is $C$-valued on Zariski generic points of $V$ over $k$.
It follows from Fact~\ref{densesharp} that {\em Kolchin generic} points of $(V,\overline s)^\sharp(K)$, that is  points not contained in any proper Kolchin closed subset over $k$, are Zariski generic in $V$.
Hence $f|_{(V,s)^\sharp(K)}$ is a $C$-valued $\delta$-rational function on $(V,\overline s)^\sharp(K)$.
Composing with the isomorphism $(V,\overline s)^\sharp(K)\to A^\sharp$, we obtain a nonconstant $C$-valued $\delta$-rational function on the Manin kernel, say $g:A^\sharp\to C$.
This gives us, at least, some nontrivial relationship between $A$ and $C$.

At this point one could invoke the fact that as $A$ is a simple abelian variety not defined over $C$, $A^\sharp$ is {\em locally modular strongly minimal} and hence {\em orthogonal} to $C$, which rules out the existence of any such $g:A^\sharp\to C$.
An explanation of these claims and their proofs can be found in~\cite[$\S$5]{marker}.
But this route uses the rather deep ``Zilber dichotomy" for differentially closed fields, which is not really required.
Instead, one can use $g$ to more or less explicitly build an isomorphism between $A$ and an abelian variety over $C$.
The existence of such an isomorphism follows from the study of finite rank definable groups in differentially closed fields, carried out by Hrushovski and Sokolovic in the unpublished manuscript~\cite{hrushovskisokolovic} and exposed in various places.
In brief: the simplicity of $A$ implies {\em semiminimality} of $A^\sharp$ (see~5.2 and~5.3 of~\cite{marker}) and then $g:A^\sharp\to C$ gives rise to a surjective $\delta$-rational group homomorphism with finite kernel $\phi:A^\sharp\to H(C)$, for some algebraic group $H$ over $C$.
(See, for example, Proposition~3.7 of~\cite{eagle} for a detailed construction of $\phi$.)
A final argument, which is explained in detail in~5.12 of~\cite{marker}, produces from $\phi$ the desired isomorphism between $A$ and an abelian variety defined over $C$.
\end{proof}

\medskip
\subsection{The Proof of Theorem~\ref{diffex}}
We now exhibit a complex affine differential algebra with the required properties.

Fix a positive transcendence degree function field $k$ over $\mathbb C$ equipped with a derivation $\delta$ so that the constant field of $(k,\delta)$ is $\mathbb C$.
For example $k$ is the rational function field $\mathbb C(t)$ and $\delta=\frac{\text d}{\text{dt}}$.
Applying the construction of the previous section to a simple abelian variety over $k$ that is not isomorphic to one defined over $\mathbb C$ we obtain a $D$-variety $(V,s)$ over $k$ satisfying the two conclusions of Fact~\ref{dvarexample}; namely, that $(V,s)^\sharp(k^{\alg})$ is Zariski dense in $V$ and the constant field of the derivation induced on the rational function field of $V$ is $\mathbb C$.
Replacing $V$ with an affine open subset, we may moreover assume that $V$ is an affine $D$-variety.

Write $k[V]=k[b]$ for some $b=(b_1,\dots,b_n)$.

\begin{lemma}
There exists a finite tuple  $a$ from $k$ such that $k=\mathbb C(a)$ and $\delta$ restricts to a derivation on $\mathbb C[a,b]$.
\end{lemma}

\begin{proof}
Let $a_1,\dots,a_\ell\in k$ be such that $k=\mathbb C(a_1,\dots,a_\ell)$.
For each $i$, let $\delta a_i=\frac{P_i(a_1,\dots,a_\ell)}{Q_i(a_1,\dots,a_\ell)}$ where $P_i$ and $Q_i$ are polynomials over $\mathbb C$.
In a similar vein, each $\delta b_j$ is a polynomial in $b$ over $\mathbb C(a_1,\dots,a_\ell)$, so let $R_j(a_1,\dots,a_\ell)$ be the product of the denominators of these coefficients.
Then let $Q(a_1,\dots,a_\ell)$ be the product of all the $Q_i$'s and the $R_j$'s.
Set $a=\big(a_1,\dots,a_\ell, \frac{1}{Q(a_1,\dots,a_\ell)}\big)$.
A straightforward calculation using the Leibniz rule shows that this $a$ works.
\end{proof}

Let $R:=\mathbb C[a,b]$.
This will witness the truth of Theorem~\ref{diffex}.
Part~(i) of that theorem is immediate from the construction: $\Frac(R)=k(V)$, and so the constant field of $\big(\Frac(R),\delta\big)$ is $\mathbb C$.

Toward part~(ii), let $X$ be the $\mathbb C$-locus of $(a,b)$ so that $R=\mathbb C[X]$.
The projection $(a,b)\mapsto a$ induces a dominant morphism $X\to Y$, where $Y$ is the $\mathbb C$-locus of $a$, such that the generic fibre $X_a=V$.
The derivation on $R$ induces a $D$-variety structure on $X$, say $s_X:X\to TX$.
(Note that as $X$ is defined over the constants the prolongation is just the tangent bundle.)
Since $\delta$ on $k[V]$ extends $\delta$ on $R$, $s_X$ restricts to $s$ on $V$.

Let $v\in(V,s)^\sharp(k^{\alg})$ and consider the $\mathbb C$-locus $Z\subset X$ of $(a,v)$.
The fact that $s(v)=\nabla(v)$ implies that $s_X(a,v)=\nabla(a,v)\in TZ$.
This is a Zariski closed condition, and so $Z$ is a $D$-subvariety of $X$.
Via the correspondence of $\S$\ref{dvarieties}, the ideal of $Z$ is therefore a prime $\delta$-ideal of $R$.
As $v$ is a tuple from $k^{\alg}=\mathbb C(a)^{\alg}$, the generic fibre $Z_a$ of $Z\to Y$ is zero-dimensional.
In particular, $Z\neq X$, and so the ideal of $Z$ is nontrivial.

But the set of such points $v$ is Zariski dense in $V=X_a$, and so the union of the associated $D$-subvarieties $Z$ is Zariski dense in $X$.
Hence the intersection of their ideals must be zero.
We have proven that the intersection of all nontrivial prime differential ideals of $R$ is zero.
This gives part~(ii).

Finally there is the question of the Krull dimension of $R$, that is, $\dim X$.
As pointed out in Remark~\ref{nadgroup}, if we choose our abelian variety in the construction to be an elliptic curve (defined over $k$ and not defined over $\mathbb C$) then $\dim V=2$, and so $\dim X=\dim Y+2$.
But $\dim Y$, which is the transcendence degree of $k$, can be any positive dimension: for any $n$ there exist transcendence degree $n$ function fields $k$ over $\mathbb C$ equipped with a derivation such that $k^\delta=\mathbb C$.
So we get examples of any Krull dimension $\geq 3$.\qed

\medskip
\section{A counterexample in Poisson algebras}
\label{sect-poissoncounterexample}
\noindent
In this section, we use Theorem \ref{diffex} to show that for each $d\ge 4$, there is a Poisson algebra of Krull dimension $d$ that does not satisfy the Poisson Dixmier-Moeglin equivalence.  To do this, we need the following lemma, which gives us a way of getting a Poisson bracket from a pair of commuting derivations.

\begin{lemma}
\label{difftopoisson}
Suppose $S$ is a ring equipped with two commuting derivations $\delta_1,\delta_2$, and $k$ is a subfield contained in the kernel of both.
Then
$$\{r,s\}:=\delta_1(r)\delta_2(s)-\delta_2(r)\delta_1(s)$$
defines a Poisson bracket over $k$ on $S$.
\end{lemma}

\begin{proof}
Clearly, $\{r,r\}=0$. The maps $\{r,-\}$ and $\{-,r\}:S\to S$ are $k$-linear derivations since they are $S$-linear combinations of the $k$-linear derivations $\delta_1$ and $\delta_2$. 
It only remains to check the Jacobi identity.
A direct (but tedious) computation shows that 
\begin{eqnarray*}
\{r,\{s,t\}\}
&=&\delta_1(r)\delta_2\delta_1(s)\delta_2(t)+ \delta_1(r)\delta_1(s)\delta_2^2(t)\\
&-& \delta_1(r)\delta_2^2(s)\delta_1(t) -\delta_1(r)\delta_2(s)\delta_2\delta_1(t)\\
&-& \delta_2(r)\delta_1^2(s)\delta_2(t)-\delta_2(r)\delta_1(s)\delta_1\delta_2(t) \\
&+& \delta_2(r)\delta_1\delta_2(s)\delta_1(t)+\delta_2(r)\delta_2(s)\delta_1^2(t).
\end{eqnarray*}
Using this and the commutativity of the derivations $\delta_1$ and $\delta_2$, one can easily check that
$\{r,\{s,t\}\}+\{t,\{r,s\}\}+\{s,\{t,r\}\}=0$, as desired.
\end{proof}

\begin{proposition} \label{lem: pb} Let $k$ be a field of characteristic zero, and $R$ an integral $k$-algebra endowed with a nontrivial $k$-linear derivation $\delta$.  Then there is a Poisson bracket $\{\, \cdot \, , \, \cdot \,\}$ on $R[t]$ with the following properties.  
\begin{enumerate}
\item The Poisson center of $\Frac\big(R[t]\big)$ is equal to the field of constants of $(R,\delta)$;
\item if $P$ is a prime differential ideal of $R$ then $PR[t]$ is a Poisson prime ideal of $R[t]$;
\end{enumerate}

\end{proposition}
\begin{proof}
Let $S=R[t]$ and consider the two derivations on $S$ given by $\delta_1(p):=p^\delta$ and $\delta_2:=\frac{d}{dt}$.
Here by $p^\delta$ we mean the polynomial obtained by applying $\delta$ to the coefficients.
Note that $\delta_1$ is the unique extension of $\delta$ that  sends $t$ to zero, while $\delta_2$ is the unique extension of the trivial derivation on $R$ that sends $t$ to $1$.
It is easily seen that $k$ is contained in the kernel of both.
These derivations commute on $S$ since $\delta_2$ is trivial on $R$ and on monomials of degree $n>0$ we have
\begin{eqnarray*}
\delta_1(\delta_2(rt^n))
&=&
\delta_1(nrt^{n-1})\\
&=&
n\delta_1(r)t^{n-1}\\
&=&
\delta_2(\delta_1(r)t^n)\\
&=&
\delta_2(\delta_1(rt^n)).
\end{eqnarray*}
Lemma~\ref{difftopoisson} then gives us that 
$$\{p(t),q(t)\}:=p^\delta(t)q'(t)-p'(t)q^\delta(t)$$
is a Poisson bracket on $S$.

Let $q(t)\in {\rm Frac}(S)$ be in the Poisson centre.
Then for $a\in R$ we must have,
$$ 0 =\left\{q(t),a\right\}=-\delta(a)q'(t). $$
Since $\delta$ is not identically zero on $R$, we see that
$\displaystyle q'(t)=0$.
This forces $q(t)=\alpha\in {\rm Frac}(R)$.
But then $0=\{\alpha,t\}=\delta(\alpha),$
and so $\alpha$ is in the constant field of $\Frac(R)$.
Conversely, if $f\in {\rm Frac}(R)$ with $\delta(f)=0$ then $\{f,q(t)\}=\delta(f)q'(t)=0$ for all $q(t)\in {\rm Frac}(R[t])$, and hence $f$ is in the Poisson center.  

If $P$ is a prime ideal of $R$ then $Q:=PS$ is a prime ideal of $S$.
If moreover $\delta(P)\subseteq P$ then $\{P,q(t)\}\subseteq q'(t)\delta(P)\subseteq Q$ for any $q(t)\in S$.
Hence
$$\{Q,q(t)\} \subseteq \{P,q(t)\}S+P\{S,q(t)\}\subseteq Q.$$
It follows that $\{Q,S\}\subseteq Q$, and so $Q$ is a Poisson prime ideal of $S$.
\end{proof}

\begin{corollary}  Let $d\ge 4$ be a natural number.  There exists a complex affine Poisson algebra of Krull dimension $d$ such that $(0)$ is Poisson rational but not Poisson locally closed.
In particular, the Poisson Dixmier-Moeglin equivalence fails.
\label{counterexample}
\end{corollary}

\begin{proof}
By Theorem \ref{diffex}, there exists a complex affine algebra $R$ of Krull dimension $d-1$ equipped with a derivation $\delta$ such that that field of constants of $\big({\rm Frac}(R),\delta\big)$ is $\mathbb{C}$ and such that the intersection of the nontrivial prime differential ideals of $R$ is zero.
By Proposition~\ref{lem: pb} we see that $R[t]$ can be endowed with a Poisson bracket such that $(0)$ is a Poisson rational ideal and such that the nontrivial prime differential ideals $P$ of $R$ generate nontrivial Poisson prime ideals $PR[t]$ in $R[t]$.
These Poisson prime ideals of $R[t]$ must then also have trivial intersection.
We have thus shown that $(0)$ is not Poisson locally closed in $R[t]$.  
\end{proof}

\medskip
\section{A finiteness theorem on height one differential prime ideals}
\label{sec:jouanolou}
\noindent
In this section we will prove the following differential-algebraic theorem, which will be used in the next section to establish a weak Poisson Dixmier-Moeglin equivalence.

\begin{theorem}
Let $A$ be an affine $\mathbb C$-algebra equipped with $\mathbb{C}$-linear derivations $\delta_1,\ldots ,\delta_m$.
If there are infinitely many height one prime differential ideals then there exists $f\in {\rm Frac}(A)\setminus\mathbb C$ with $\delta_i(f)=0$ for all $i=1,\ldots,m$.
\label{thm: main}
\end{theorem}

When $m=1$ this theorem is a special case of unpublished work of Hrushovski~\cite[Proposition~2.3]{hrushovski-jouanolou}.
It is possible that Hrushovski's method (which goes via a generalisation of a theorem of Jouanolou) extends to this setting of several (possibly noncommuting) derivations.
But we will give an algebraic argument that is on the face of it significantly different.
We first show that if the principal ideal $fA$ is already a differential ideal then $\delta(f)/f$ is highly constrained (Proposition~\ref{prop: W}).
We then use this, together with B\'ezout-type estimates (Proposition~\ref{lem: bigbound}), to deal with the case when the given height one prime differential ideals are principal (Proposition~\ref{prop: space}).
Finally, using Mordell-Weil-N\'eron-Severi, we are able to reduce to that case.

We will use the following fact from valued differential fields.

\begin{fact}{\cite[Corollary~5.3]{morrison}\footnote{We thank Matthias Aschenbrenner for pointing us to this reference.}}
\label{lem: bound}
Suppose $K/k$ is a function field of characteristic zero and transcendence degree $d$, and $v$ is a rank one discrete valuation on $K$ that is trivial on $k$ and whose residue field is of transcendence degree $d-1$ over $k$.
Then for any $k$-linear derivation $\delta$ on $K$ there is a positive integer $N$ such that $v\big(\frac{\delta(f)}{f}\big)>-N$ for all nonzero $f\in K$.
\end{fact}

\begin{proposition} Let $k$ be a field of characteristic zero and let $A$ be a finitely generated integrally closed $k$-algebra equipped with a $k$-linear derivation $\delta$.  Then there is a finite-dimensional $k$-vector subspace $W$ of $A$ with the property that whenever $f\in A\setminus \{0\}$ has the property that $\delta(f)/f\in A$ we must have $\delta(f)/f\in W$.
\label{prop: W}
\end{proposition}

\begin{proof} We have that $A$ is the ring of regular functions on some irreducible affine normal variety $X$.  Moreover, $X$ embeds as a dense open subset of a projective normal variety $Y$.  Then $Y\setminus X$ is a finite union of closed irreducible subsets whose dimension is strictly less than that of $X$.  We let $Y_1,\ldots ,Y_\ell$ denote the closed irreducible subsets in $Y\setminus X$ that are of codimension one in $Y$.
Let $f\in A\setminus \{0\}$ be such that $g:=\delta(f)/f\in A$.  Then $g$ is regular on $X$ and so its poles are concentrated on $Y_1,\dots, Y_\ell$.
But by Fact \ref{lem: bound}, if we let $\nu_i$ be the valuation on $k(X)$ induced by $Y_i$, we have that there is some natural number $N$ that is independent of $f$ such that
$\nu_i(g)\ge -N$ for all $i=1,\ldots,\ell$.
It follows that 
$$g\in W:=\{ s\in k(Y)\setminus \{0\}~ :~ {\rm div}(s)\ge -D\}\cup \{0\},$$ where $D$ is the effective divisor $N[Y_1]+\cdots +N[Y_\ell]$.
Since $Y$ is a projective variety that is normal in codimension one, $W$ is a finite-dimensional $k$-vector subspace of ${\rm Frac}(A)$, see \cite[Corollary A.3.2.7]{SilvermanHindry}.  By assumption $g=\delta(f)/f\in A$ and so we may replace $W$ by $W\cap A$ if necessary to obtain a finite-dimensional subspace of $A$.   
\end{proof}

\medskip
\subsection{B\'ezout-type estimates}
The next step in our proof of Theorem~\ref{thm: main} is Proposition~\ref{lem: bigbound}
below, which has very little to do with differential algebra at all -- it is about linear operators on an affine complex algebra.
Its proof will use estimates that we derive in Lemma~\ref{lem: generalintersection} on the number of solutions to certain systems of polynomial equations over the complex numbers, given that the system has only finitely many solutions.

We will use the following B\'ezout inequality from intersection theory.
It is well known, in fact, that in the statement below one can replace $N^{d+1}$ by $N^d$, but we are unaware of a proper reference and the weaker bound that we give is sufficient for our purposes.

\begin{fact}\label{intersection}
Suppose $X\subseteq \mathbb C^d$ is the zero set of a system of polynomial equations of degree at most $N$.
Then the number of zero-dimensional irreducible components of $X$ is at most $N^{d+1}$.
\end{fact}
\begin{proof}
We define the degree, ${\rm deg}(Y)$, of an irreducible Zariski closed subset $Y$ of $\mathbb{C}^d$ of dimension $r$ to be the supremum of the number of points in $Y\cap H_1\cap \cdots \cap H_r$, where $H_1,\ldots ,H_r$ are $r$ affine hyperplanes with the property that 
$Y\cap H_1\cap \cdots \cap H_r$ is finite.
In general, the degree of a Zariski closed subset $Y$ of $\mathbb{C}^d$ is defined to be the sum of the degrees of the irreducible components of $Y$.
In particular, if $f(x_1,\ldots ,x_d)\in \mathbb{C}[x_1,\ldots ,x_d]$ has total degree $D$ then the hypersurface $V(f)$ has degree $D$, and a point has degree one.
If $Y$ and $Z$ are Zariski closed subsets of $\mathbb{C}^d$ then ${\rm deg}(Y\cap Z)\le {\rm deg}(Y)\cdot {\rm deg}(Z)$, see Heintz~\cite[Theorem 1]{Hei}.

Now let $f_1,\ldots,f_{s}\in \mathbb{C}[x_1,\ldots,x_d]$ be of degree at most $N$ such that $X=V(f_1\ldots,f_{s})$.
By Kronecker's Theorem\footnote{We could not find a very good reference for Kronecker's Theorem in this form, but it can be seen as a special case of Ritt's~\cite[Chapter~VII, $\S$17]{ritt}.
Michael Singer pointed this out to us in a private communication in which he has also supplied a direct proof.}
there are $g_1,\ldots,g_{d+1}\in \mathbb{C}[x_1,\ldots,x_d]$ which are $\mathbb{C}$-linear combinations of the $f_i$'s such that $V(g_1,\ldots,g_{d+1})=V(f_1,\ldots,f_s)$.
In particular, $X=V(g_1)\cap \cdots \cap V(g_{d+1})$ has degree at most $N^{d+1}$ and so the number of zero dimensional components of $X$ is at most $N^{d+1}$. 
\end{proof}

The following is an easy exercise on the Zariski topology of $\mathbb{C}^d$.

\begin{lemma}
Suppose $Y$ and $Z$ are Zariski closed sets in $\mathbb{C}^d$
and suppose that $Y\setminus Z$ is finite.  Then $|Y\setminus Z|$ is bounded by the number of zero-dimensional irreducible components of $Y$.
\label{lem: YZ}
\end{lemma}
\begin{proof}
We write $Y=Y_1\cup \cdots \cup Y_m$, with $Y_1,\ldots ,Y_m$ irreducible and $Y_i\not\subseteq Y_j$ for $i\neq j$.  Then 
$$Y\setminus Z = (Y_1\setminus Z)\cup \cdots \cup (Y_m\setminus Z).$$ 
Since $Y_i\cap Z$ is a Zariski closed subset of $Y_i$, it is either equal to $Y_i$ or it has strictly smaller dimension than $Y_i$.  In particular, if $Y_i$ is positive dimensional, then $Y_i\setminus Z$ must be empty, since otherwise $Y\setminus Z$ would be infinite.  Thus $|Y\setminus Z|\le |\{ i\colon Y_i~{\rm is~a~point}\}|$ and so the result follows.
\end{proof}

This is the main counting lemma.

\begin{lemma}
\label{lem: generalintersection}
Let $n$, $d$, and $N$ be natural numbers and suppose that $X\subseteq \mathbb{C}^{n+d}$ is the zero set of a system of polynomials of the form
$$\sum_{i=1}^n P_i(y_1,\ldots ,y_d)x_i + Q(y_1,\ldots ,y_d),$$ where $P_1,\ldots ,P_n,Q\in\mathbb{C}[y_1,\ldots ,y_d]$ are polynomials of degree at most $N$.  If $X$ is finite then $|X|\le \big((n+1)N\big)^{d+1}$.
\end{lemma}

\begin{remark}
We will be using this lemma in a context where $N=1$ and $d$ is fixed.
So the point is that the bound grows only polynomially in $n$.
\end{remark}

\begin{proof}[Proof of Lemma~\ref{lem: generalintersection}]
Let the defining equations of $X$ be
$$\sum_{i=1}^n P_{i,j}(y_1,\ldots ,y_d)x_i + Q_j(y_1,\ldots ,y_d)$$ for $j=1,\ldots ,m$ and let $\pi: \mathbb{C}^{n+d}\to \mathbb{C}^d$ be the map
$$(x_1,\ldots ,x_n,y_1,\ldots ,y_d)\mapsto (y_1,\ldots ,y_d).$$
So $X_0:=\pi(X)$ is the set of points $\alpha=(\alpha_1,\ldots ,\alpha_d)\in \mathbb{C}^d$ for which the system of equations
\begin{equation}
\label{system}
\sum_{i=1}^n P_{i,j}(\alpha)x_i + Q_j(\alpha) = 0
\end{equation}
for $j=1,\ldots ,m$, has a solution.
Now suppose $X$ is finite.
Then for $\alpha\in X_0$, $\pi^{-1}(\alpha)\cap X$ must be a single point as it is a finite set defined by affine linear equations.
So $|X|=|X_0|$, and it suffices to count the size of $X_0$.
Moreover, $X_0$ is precisely the set of $\alpha$ such that~(\ref{system}) has a unique solution.

Note that if $n>m$ then for every $\alpha$, either~(\ref{system}) has no solution or it has infinitely many solutions.
So, assuming that $X_0$ is nonempty, we may assume that $n\leq m$.
But if $n=m$ then off a proper Zariski closed set of $\alpha$ in $\mathbb C^d$ the system~(\ref{system}) has a unique solution -- contradicting that $X_0$ is finite.
So $n<m$.

Let $A(y_1,\ldots ,y_d)$ denote the $m\times (n+1)$ matrix whose $j$-th row is the row
$$[P_{1,j}(y_1,\ldots ,y_d),~\ldots ~,P_{n,j}(y_1,\ldots ,y_d), ~-Q_j(y_1,\ldots ,y_d)]$$ and let 
$B(y_1,\ldots ,y_d)$ denote the $m\times n$ matrix obtained by deleting the $(n+1)$-st column of $A$.
We have that $\alpha\in X_0$ if and only if the last column of $A(\alpha)$ is in the span of the column space of $B(\alpha)$; equivalently, $A(\alpha)$ and $B(\alpha)$ must have the same rank and this rank is necessarily $n$.

Let $Y$ be the set of all $\alpha$ such that every $(n+1)\times(n+1)$ minor of $A(\alpha)$ vanishes.
So $\alpha\in Y$ says that the rank of $A(\alpha)$ is $\leq n$.
Let $Z$ be the set of all $\alpha$ such that every $n\times n$ minor of $B(\alpha)$ vanishes.
So $\alpha\in Z$ means the rank of $B(\alpha)< n$.
Hence $X_0=Y\setminus Z$.
Since each $(n+1)\times (n+1)$ minor of $A(y_1,\ldots ,y_d)$ has degree at most $(n+1)N$, we see from Fact~\ref{intersection} that the number of zero-dimensional irreducible components of $Y$ is at most $((n+1)N)^{d+1}$.
By Lemma~\ref{lem: YZ}, $|X_0|\leq((n+1)N)^{d+1}$.
\end{proof}

We now give the main conclusion of this subsection.  In order to obtain the desired estimates, we will work with products of vector spaces and so we give some notation.  Given a field $k$ and an associative $k$-algebra $A$ if $V$ and $W$ are $k$-vector subspaces of $A$, we define $VW$ to be the span of all products $vw$ with $v\in V$ and $w\in W$.  Since $A$ is associative, it is easily checked that $(VW)U=V(WU)$ for subspaces $V$, $W$, and $U$ of $A$.   We may thus write $VWU$ unambiguously and so in the case that $V$ is a vector space and $n\ge 1$, we then take $V^n$ to be $V\cdot V\cdots V$, where there are $n$ copies of $V$ inside the product.    

\begin{proposition}
Suppose $A$ is an affine $\mathbb C$-algebra, $L_1,\ldots, L_m$ are $\mathbb C$-linear operators on $A$, and $V$ and $W$ are finite-dimensional $\mathbb C$-linear subspaces of $A$.
Let $X$ be the set of $f\in V$ for which $\frac{L_j(f)}{f}\in W$ for all $j=1,\ldots ,m$.
Then the image of $X$ in the projectivisation $\mathbb{P}(V)$ is either uncountable or has size at most $(\dim V)^{2+m\dim W}$.
\label{lem: bigbound}
\end{proposition}

\begin{remark}
One should think here of $W$ as fixed and $V$ as growing.
So the Proposition gives a bound that grows only polynomially in $\dim V$.
\end{remark}

\begin{proof}[Proof of Proposition~\ref{lem: bigbound}]
Let $\{r_1,\ldots, r_n\}$ be a basis for $V$ and let $\{s_1,\ldots ,s_d\}$ be a basis for $W$.
We are interested in the set $\mathcal{T}$ of $(x_{1},\ldots ,x_{n})\in \mathbb{C}^n$ for which there exists $(y_{1,j},\ldots ,y_{d,j})\in \mathbb{C}^d$, for $j=1,\ldots ,m$, such that
$\displaystyle
L_j\big(\sum_{i=1}^n x_{i} r_i \big) = \big(\sum_{i=1}^n x_{i} r_i \big)(y_{1,j}s_1+\cdots +y_{d,j}s_d)$.
Since the $L_j$ are linear, this becomes 
\begin{equation}
\sum_{i=1}^n x_i L_j(r_i) = \big(\sum_{i=1}^n x_{i} r_i \big)(y_{1,j}s_1+\cdots +y_{d,j}s_d).
\label{eq: system}
\end{equation}
We point out that if $(x_1,\ldots ,x_n)\in \mathcal{T}$ then so is $(\lambda x_1,\ldots ,\lambda x_n)$ for $\lambda$ in $\mathbb{C}$.
As we are only interested in solutions in $\mathbb{P}(\mathbb{C}^n)$, we will let $\mathcal{T}_q$ denote the set of elements $(x_1,\ldots ,x_n)$ in $\mathcal{T}$ with $x_q=1$ and we will bound the size of each $\mathcal{T}_q$.

Note that if $(x_1,\ldots ,x_n)\in \mathcal{T}_q$ for some $q$, then as  $A$ is an integral domain, $\big(\sum_{i=1}^n x_{i} r_i \big)\neq 0$, and $s_1,\ldots ,s_d$ is a basis for $W$, we see that for $j=1,\ldots ,m$ there is necessarily a unique solution $(y_{1,j},\ldots ,y_{d,j})\in \mathbb{C}^d$ such that Equation (\ref{eq: system}) holds.
So the cardinality of $\mathcal{T}_q$ is the same as the set of solutions to Equations (\ref{eq: system}) in $\mathbb C^{n+md}$ with $x_q=1$.
It is this latter set that we count.

Let $w_1,w_2,\ldots ,w_{\ell}$ be a basis for $\span_{\mathbb C}\big(VW\cup\bigcup_{j=1}^mL_j(V)\big)$.
We thus have expressions 
$L_j(r_i) = \sum_{i,j,p} \alpha_{i,j,p} w_p$ and 
$r_is_k = \sum_{i,p,k} \beta_{i,k,p} w_p$ for $j=1,\ldots ,m$, $i=1,\ldots ,n$, $k=1,\ldots ,d$.
Combining these expressions with Equation (\ref{eq: system}), we see that for $j\in \{1,\ldots ,m\}$, we have
$$\sum_{i=1}^n x_i\left(\sum_{p} \alpha_{i,j,p} w_p \right) = \sum_{i,k} x_i y_{k,j} \left( \sum_p \beta_{i,k,p} w_p\right).$$
In particular, if we extract the coefficient of $w_p$, we see that for $j\in \{1,\ldots ,m\}$ and $p\in \{1,\ldots ,\ell\}$, we have
$\displaystyle \sum_{i=1}^n  \alpha_{i,j,p} x_i = \sum_{i,k} x_i y_{k,j} \beta_{i,k,p}$.
Imposing the condition that $x_q=1$ we obtain the system of equations
\begin{equation*}
\label{eq: newsystem2}
\alpha_{q,j,p}+ \sum_{i\neq q} \alpha_{i,j,p} x_i- \sum_{k} y_{k,j}\beta_{q,k,p}-  \sum_{i\neq q }\sum_{k}  x_i y_{k,j} \beta_{i,k,p} = 0.
\end{equation*}
for $j=1,\ldots ,m$ and $p=1,\ldots ,\ell$.
This system of equations can be described as affine linear equations in $\{x_1,\ldots ,x_n\}\setminus \{x_q\}$ whose coefficients are polynomials in $y_{k,j}$, $1\le k\le d$, $1\le j\le m$ of total degree at most one, and hence by Lemma~\ref{lem: generalintersection} the number of solutions is either infinite---in which case it is uncountable as it has a component of dimension bigger than or equal to one and we are working over $\mathbb C$---or is at most $n^{md+1}$.
Thus the size of the union of $\mathcal{T}_q$ as $q$ ranges from $1$ to $n$ is either uncountable or has size at most $n^{md+2}$, as desired.
\end{proof}

\medskip
\subsection{The case of principal ideals}
\label{ufd}
Here we deal with the case of Theorem~\ref{thm: main} when there are infinitely many principal prime differential ideals.

\begin{proposition}
Let $A$ be an integrally closed affine $\mathbb{C}$-algebra with $\mathbb C$-linear derivations $\delta_1,\ldots ,\delta_m$.
Suppose that there exists an infinite set of elements $r_1,r_2,\ldots $ of $A$ such that $\delta_j(r_i)/r_i\in A$ for $j=1,\ldots ,m$ and $i\ge 1$ and such that their images in $(A\setminus \{0\})/\mathbb{C}^*$ generate a free abelian semigroup.
Then the field of constants of $\big(\Frac(A),\delta_1,\dots,\delta_m\big)$ is strictly bigger than $\mathbb C$.
\label{prop: space}
\end{proposition}

\begin{proof}
Denote the multiplicative semigroup of $A\setminus \{0\}$ generated by $r_1,r_2,\ldots $ by~$\mathcal T$.
As the operator $x\mapsto\delta_j(x)/x$ transforms multiplication into addition, we have that $\delta_j(r)/r\in A$ for all $r\in \mathcal{T}$ and $j\in \{1,\ldots ,m\}$.
By Proposition \ref{prop: W} there is thus a finite-dimensional subspace $W$ of $A$ such that $\delta_j(r)/r\in W$ for all $r\in \mathcal{T}$ and $j\in \{1,\ldots ,m\}$.

Let $q:=(1+{\rm Kdim}(A))(2+m{\rm dim}(W))$, where ${\rm Kdim}(A)$ is the Krull dimension of $A$.  
We pick a finite-dimensional vector subspace $U$ of $A$ that contains $r_1,\ldots ,r_{q+1}$.
We claim that for $N$ sufficiently large the image of
$$\mathcal{X}_N:= \{r\in U^{N(q+1)}:\delta_j(r)/r\in W, j=1,\dots,m\}$$
in $\mathbb{P}(U^{N(q+1)})$ is uncountable.
Indeed, if it were not, then by Proposition~\ref{lem: bigbound} its size would be bounded by
$\displaystyle \left({\rm dim}(U^{N(q+1)})\right)^{2+m\dim W}$.
Basic results in Gelfand-Kirillov dimension (see \cite[Theorem 4.5 (a)]{krauselenagan}) give that
${\rm dim}(U^{N(q+1) } )< (N (q+1))^{1+\rm Kdim(A)} $ for all $N$ sufficiently large.
Hence by choice of $q$ we get that the size of the image of $\mathcal{X}_N$ in $\mathbb{P}(U^{N(q+1)})$ is eventually at most
$\displaystyle (N(q+1))^q$.
On the other hand, for each $0\le i_1,\ldots ,i_{q+1}\le N$ we have
$r_1^{i_1}\cdots r_{q+1}^{i_{q+1}}\in U^{N(q+1)}\cap \mathcal{T}$, and by assumption these give rise to distinct elements of $\mathcal{X}_N$ whose images in $\mathbb{P}(U^{N(q+1)})$ are also distinct.
So the size of the image of $\mathcal{X}_N$ in $\mathbb{P}(U^{N(q+1)})$ is at least $(N+1)^{q+1}$.
Comparing the degrees of these polynomials in $N$, gives a contradiction for large $N$.

Thus, fixing $N$ sufficiently large, setting $V:=U^{N(q+1)}$, and
$$\mathcal{X}:=\{r\in V:r\neq 0,\delta_j(r)/r\in A, j=1,\dots,m\},$$
we have shown that the image of $\mathcal{X}$ in $\mathbb{P}(V)$ is uncountable.
Let $\mathcal{S}$ denote the set of all ideals of the form $rA$ where $r\in \mathcal{X}$.
We claim that Lemma \ref{lem: diffcenter} applies to $\mathcal S$, giving us the sought for differential constant $f\in\Frac(A)\setminus\mathbb C$ which would complete the proof of the proposition.
Indeed, condition~(1), that each $I\in\mathcal S$ is differential, holds because $I=rA$ with $\delta_j(r)/r\in A$ for all $j=1,\dots,m$.
Condition~(3), that each ideal in $\mathcal S$ has nontrivial intersection with the finite-dimensional space $V$, holds by construction: each $I\in \mathcal S$ is generated by an element of $V$.
It remains only to prove condition~(2); that 
$\bigcap\mathcal{S}  =(0)$.

To see this, note that $A$ is the ring of regular functions on some irreducible affine normal variety $X$, and $X$ embeds as a dense open subset of a projective normal variety $Y$.  Let $Z_1,\ldots ,Z_s$ denote the irreducible components of $Y\setminus X$ of codimension one.
For every $f\in A$, the negative part of the $\Div(f)$ is supported on $\{Z_1,\ldots ,Z_s\}$.
Suppose, toward a contradiction, that there is a nonzero $a\in rA$ for all $r\in\mathcal X$.
If  $\{V_1,\ldots ,V_t\}$ is the support of the positive part of $\Div(a)$, then the positive part of $\Div(r)$ is supported on $\{V_1,\dots,V_t\}$ also, for all $r\in \mathcal X$.
So for all $r\in\mathcal X$, $\Div(r)$ is supported on $\{Z_1,\ldots ,Z_s,V_1,\ldots ,V_t\}$.
But there are only countably many divisors supported on this finite set.
If two nonzero elements of $\mathbb{C}(Y)$ have the same associated divisor then their ratio is regular on $Y$ and hence necessarily in $\mathbb{C}^*$.
It follows that the image of $\mathcal{X}$ in $\mathbb{P}(V)$ is necessarily countable, a contradiction.  
\end{proof}

\subsection{The Proof of Theorem \ref{thm: main}}
\label{redufd}
In this section we prove Theorem \ref{thm: main}.  To do this, we need a lemma that shows we can reduce to the principal case.  This lemma appears to be something that should be in the literature, but we have not encountered this result before.
\begin{lemma} Let $k$ be a finitely generated extension of $\mathbb{Q}$, let $A$ be a finitely generated commutative $k$-algebra that is a domain.  Then there is a nonzero $s\in A$ such that $A_s=A[1/s]$ is a unique factorization domain.
\label{lem: 1}
\end{lemma}
\begin{proof} 
We recall that a noetherian integral domain $A$ is a UFD if and only if $X = {\rm Spec}(A)$ is normal
and ${\rm Cl}(X) = 0$ \cite[II.6.2]{Har}.  By replacing $A$ by $A[1/f]$ for some nonzero $f\in A$ we may assume that $A$ is integrally closed.  Let $X={\rm Spec}(A)$.  Note that $X$ is quasi-projective and hence is an open subset of an irreducible projective scheme $Y$.  We may pass to the normalization of $Y$ if necessary (this does not affect $X$) and assume that $X$ is an open subset of a normal projective scheme $Y$.  Note that $Y$ is noetherian, integral, and separated and so ${\rm Cl}(Y)$ surjects on ${\rm Cl}(X)$ \cite[Proposition II.6.5]{Har}.  From a version of the Mordell-Weil-N\'eron-Severi theorem (see \cite[Corollary 6.6.2]{Lang} for details), we see that ${\rm Cl}(Y)$ is a finitely generated abelian group, and so ${\rm Cl}(X)$ must be too.

It follows that there exist height one prime ideals $P_1,\ldots, P_r$ of $A$ such that if $P$ is a height one prime ideal of $A$ then
there are integers $a_1,\ldots ,a_r$ such that $[V(P)]=\sum_{i=1}^r a_i [V(P_i)]$ in ${\rm Cl}(X)$, where for a height one prime $Q$, $[V(Q)]$ denotes the image of the irreducible subscheme of $X$ that corresponds to $Q$ in ${\rm Cl}(X)$.   Let $s$ be a nonzero element of $P_1\cap P_2\cap \cdots \cap P_r$.  Then the equation
$[V(P)]=\sum_{i=1}^r a_i [V(P_i)]$ gives that $P \prod_{\{i : a_i<0\}} P_i^{a_i} = f \prod_{\{i : a_i>0\}} P_i^{a_i}$ for some nonzero rational function $f$.  Then, passing to the localization $A_s$ we see that
$$P_s = \left(P \prod_{\{i : a_i<0\}} P_i^{-a_i}\right) \otimes_A A_s=\left(f \prod_{\{i : a_i>0\}} P_i^{a_i}\right)\otimes_A A_s= (fA)_s,$$ where we regard $fA$ as a fractional ideal.  Since $P_s\subseteq A_s$, we see that $f\in A_s$ and so $P_s=fA_s$ is principal for each height one prime ideal $P$ of $A$.  It follows that all height one primes of $A_s$ are principal and hence $A_s$ is a unique factorization domain.
\end{proof}

We are finally ready to prove Theorem \ref{thm: main}.

 \begin{proof}[Proof of Theorem \ref{thm: main}]
We have an affine $\mathbb C$-algebra $A$ with $\mathbb C$-derivations $\delta_1,\dots,\delta_m$, and with infinitely many height one prime differential ideals.
Suppose, toward a contradiction, that the field of contants of $\big(\Frac(A),\delta_1,\dots,\delta_m\big)$ is $\mathbb C$.

Note that the derivations extend uniquely by the quotient rule to any localisation $A[\frac{1}{f}]$, and that since any such $f$ can only be contained in finitely many height one primes, this localisation also has infinitely many height one prime differential ideals.
Therefore, localising appropriately, we may assume that $A$ is integrally closed.

Next we write $A$ in the form $A_0\otimes_k\mathbb C$ for some finitely generated subfield $k$ of $\mathbb C$, and an affine $k$-subalgebra $A_0$ of $A$ such that the $\delta_j$ restrict to $k$-linear derivations on $A_0$.
This can be accomplished as follows:
write $A$ as a quotient of a polynomial ring $A=\mathbb{C}[t_1,\ldots ,t_d]/I$ where $I=\langle f_1,\dots, f_r\rangle$.
So the $x_i:=t_i+I$ generate $A$ as a $\mathbb C$-algebra.
For each $i,j$ we have $\delta_j(x_i)=q_{i,j}(x_1,\ldots ,x_d)$ for some polynomials $q_{i,j}\in\mathbb{C}[t_1,\ldots ,t_d]$.
Let $k$ denote the field generated by the coefficients of $f_1,\ldots ,f_r$ and by the coefficients of the $q_{i,j}$, and set $A_0:=k[x_1,\dots,x_d]$.

We may assume that ${\rm Frac}(A_0)\cap \mathbb{C}=k$.
Indeed, let $K:={\rm Frac}(A_0)\cap \mathbb{C}$.
Since $K$ is a subfield of the finitely generated field ${\rm Frac}(A_0)$, we have by~\cite[Theorem 11]{Vamos} that $K$ is finitely generated.
We can now replace $k$ by $K$, and so $A_0$ by $K[x_1,\dots,x_d]$.

Next we argue that $A_0$ has infinitely many height one prime differential ideals.
This will use our assumption that the field of constants of $\Frac(A)$ is just $\mathbb C$.

We claim that if $P$ is a nonzero prime differential ideal of $A$ then $P\cap A_0$ is also nonzero.
To see this, we pick $0\neq y=\sum_{i=1}^e a_i\otimes \lambda_i \in P$ with $a_1,\ldots ,a_e\in A_0$ nonzero, $\lambda_1,\dots,\lambda_e\in\mathbb C$ nonzero, and $e$ minimal.
If $e=1$ then we have $y\cdot \lambda_1^{-1}\in A_0\cap P$ and there is nothing to prove.
Assume $e>1$.
We have
$\delta_j(y)=\sum_{i=1}^e \delta_j(a_i)\otimes \lambda_i \in P$ for $j=1,\ldots ,m$.   This gives
$$\sum_{i=1}^e (a_i \delta_j(a_e)-a_e\delta_j(a_i))\otimes \lambda_i = \delta_j(a_e)y-a_e\delta_j(y)\in P.$$ 
Since the $i=e$ term above is zero, the minimality of $e$ implies that $\delta_j(a_e)y-a_e\delta_j(y)$ must be zero.
So we have that
$\delta_j(y a_e^{-1})=0$ for all $j=1,\ldots m$.
By assumption, $y=\gamma a_e$ for some $\gamma\in\mathbb C$.
So $a_e\in P\cap A_0$, as desired.

Suppose $P$ is a height one prime differential ideal in $A$.
Then $P\cap A_0$ is a prime differential ideal in $A_0$.
Since it is nonzero, it has height at least one.
To see that $P\cap A_0$ has height one, suppose that there is some nonzero prime ideal $Q$ of $A_0$ with $Q\subsetneq P\cap A_0$.  Then $QA\cap A_0=Q$ since $A$ is a free $A_0$-module.  If we now look at the set $\mathcal{I}$ of ideals $I$ of $A$ with $QA\subseteq I\subseteq P$ such that $I\cap A_0=Q$, then $\mathcal{I}$ is non-empty since $QA$ is in $\mathcal{I}$.  It follows that $\mathcal{I}$ has a maximal element, $J$.  Then $J$ is a nonzero prime ideal of $A$ that is strictly contained in $P$, contradicting the fact that $P$ has height one.
Hence $P\cap A_0$ has height one.

Moreover, if $P$ is a height one prime differential ideal in $A$ then $P$ is a minimal prime containing $(P\cap A_0)A$, so only finitely many other prime differential ideals in $A$ can have the same intersection with $A_0$ as $P$.
So the infinitely many height one prime differential ideals in $A$ give rise to infinitely many height one prime differential ideals in $A_0$.

By Lemma \ref{lem: 1} there is some nonzero $s\in A_0$ such that $B:=A_0[\frac{1}{s}]$ is a UFD.
As before, the infinitely many height one prime differential ideals of $A_0$ give rise to infinitely many height one prime differential ideals of the localisation $B$.
But as $B$ is a UFD, these ideals are principal.
We obtain an infinite set of pairwise coprime irreducible elements $r_1,r_2,\ldots $ of $B$ such that $\delta_j(r_i)/r_i\in B$ for $j=1,\ldots ,m$ and $i\ge 1$.
We now note that $B\subseteq A[1/s]$.  Furthermore, the images of the $r_i$ necessarily generate a free abelian semigroup in $(A[1/s]\setminus \{0\})/\mathbb{C}^*$, since if some non-trivial product of the $r_i$ were in $\mathbb{C}^*$ then it would be in $B\cap \mathbb{C}^*\subseteq {\rm Frac}(A_0)\cap \mathbb{C}^*=k^*$, which is impossible since the $r_i$ are pairwise coprime elements of the UFD $B$. 
Proposition~\ref{prop: space} now applies to $A[\frac{1}{s}]$ (which is integrally closed as $A$ is), and gives an $f\in {\rm Frac}(A[1/s])\setminus \mathbb{C}$ such that $\delta_j(f)=0$ for $j=1,\ldots ,m$.
But as ${\rm Frac}(A[1/s])={\rm Frac}(A)$,
this contradicts our assumption on $A$.
\end{proof}

\medskip
\section{A weak Poisson Dixmier-Moeglin equivalence}
\label{sec: weak}
\noindent
We now show that while the Poisson Dixmier-Moeglin equivalence need not hold in general, a weaker variant does hold.

\begin{theorem}
Let $A$ be a complex affine Poisson algebra.
For a Poisson prime ideal $P$ of $A$, the following are equivalent:
\begin{enumerate}
\item $P$ is rational;
\item $P$ is primitive;
\item the set of Poisson prime ideals $Q\supseteq P$ with ${\rm ht}(Q)={\rm ht}(P)+1$ is finite.
\end{enumerate}
\label{thm: weak}
\end{theorem}

\begin{proof}
We have already have shown the equivalence of (1) and (2).  It remains to prove the equivalence of (1) and (3).  By replacing $A$ by $A/P$ if necessary, we may assume that $P=(0)$.  Note that if (1) does not hold then we have a non-constant $f\in {\rm Frac}(A)$ in the Poisson centre.  We show that (3) cannot hold; the level sets of $f$ over $\mathbb C$ will give rise to infinitely many height one Poisson primes.
We write $f=a/b$ with $a,b\in A$ with $b\neq 0$.  Let $B$ be the localisation $A_b$.  Then it is sufficient to show that there are infinitely many prime ideals in $B$ of height one that are Poisson prime.
For each $\lambda\in \mathbb{C}$, we have a Poisson ideal $I_{\lambda}:=(a/b-\lambda )B$.
Since $f$ is non-constant, for all but finitely many $\lambda\in \mathbb C$, $I_\lambda$ is a proper principal ideal.
By Krull's principal ideal theorem, we have a finite set of height one prime ideals above $I_{\lambda}$, each of which is a Poisson prime ideal by Lemma \ref{lem: decomp}.
We note that if a prime ideal $P$ contains $I_{\alpha}$ and $I_{\beta}$ for two distinct complex numbers $\alpha$ and $\beta$ then $P$ contains $\alpha-\beta$, which is a contradiction.  It follows that $B$ has an infinite set of height one 
Poisson prime ideals and so (3) does not hold.  

Conversely, suppose that (1) holds.  Let $x_1,\ldots ,x_m$ be generators for $A$ as a $\mathbb{C}$-algebra, and consider the derivations $\delta_i(y)=\{y,x_i\}$.
The rationality of $(0)$ means that the constant field of $\big(\Frac(A),\delta_1,\dots,\delta_m\big)$ is $\mathbb C$, see statement~(iv) of Section~\ref{sect-diffpoi}.
It follows by Theorem~\ref{thm: main} that there are only finitely many height one prime differential ideals of $A$.
Hence there are only finitely many height one prime Poisson ideals of $A$, as desired.
\end{proof}

As a corollary we will show that the Poisson-Dixmier Moeglin equivalence holds in dimension~$\leq 3$.
But first a lemma which says that the ``Poisson points and curves'' are never Zariski dense.

\begin{lemma}
Let $A$ be a complex affine Poisson algebra of Krull dimension $d$ on which the Poisson bracket is not trivial.
Then the intersection of the set of Poisson prime ideals of height $\ge d-1$ is not trivial.
\label{lem: int}
\end{lemma}

\begin{proof}
We claim that every Poisson prime ideal of height at least $d-1$ must contain $\{a,b\}$ for all $a, b\in A$.
Let $P$ be a Poisson prime of height $\ge d-1$ and suppose, towards a contradiction, that $\{a,b\}\not \in P$.
Now, since $A/P$ has Krull dimension at most one, the morphism $\spec(A/P)\to\mathbb A_\mathbb C^2$ that is dual to the ring homomorphism given by the composition $\mathbb{C}[a,b]\hookrightarrow A\to A/P$ is not dominant.
That is, there is $0\neq f\in \mathbb{C}[x,y]$ such that $f(a,b)\in P$.  We now claim that there is some nonzero polynomial $h(x)$ such that $h(a)\in P$.  To see this, observe that if $f(x,y)$ is a polynomial in $x$ then there is nothing to prove; otherwise, it has non-constant partial derivative with respect to $y$.  Applying the derivation $\{a,-\}$ gives $\frac{\partial{f}}{\partial y}(a,b)\{a,b\}\in P$.
Since $\{a,b\}\not\in P$, we get that $\frac{\partial{f}}{\partial y}(a,b)\in P$.
Iterating if necessary, we then see that there is a nonzero polynomial $h(x)\in \mathbb{C}[x]$ such that $h(a)\in P$, as claimed.
Now $h(x)$ cannot be constant since it is nonzero and $h(a)\in P$ and $P$ is proper.  Therefore $h(x)$ splits into linear factors.  Since $P$ is a prime ideal, we see that there is some $\lambda\in\mathbb{C}$ such that $a-\lambda\in P$.  But now we apply the operator $\{-,b\}$ to get that $\{a,b\}\in P$, which is a contradiction.  Thus every Poisson prime of height $\ge d-1$ contains $\{a,b\}$ for all $a,b\in A$. Since the Poisson bracket is not trivial, the result follows. 
\end{proof}

\begin{theorem} Let $A$ be a complex affine Poisson algebra of Krull dimension $\le 3$.
Then the Poisson Dixmier-Moeglin equivalence holds for $A$.
\label{lethree}
\end{theorem}

\begin{proof}
In light of \cite[1.7(i), 1.10]{Oh} and Theorem \ref{thm: ratprim}, it is sufficient to show that if $P$ is a Poisson rational prime ideal of $A$ then $P$ is Poisson locally closed.  By replacing $A$ by $A/P$ if necessary, we may assume that $P=(0)$.  By Theorem \ref{thm: weak}, there are finitely many height one prime ideals of $A$ that are Poisson prime ideals.   By Lemma \ref{lem: int}, the intersection of prime ideals of height $\ge 2$ of $A$ that are Poisson prime ideals is nonzero.  It follows that the intersection of the nonzero Poisson prime ideals of $A$ is nonzero and hence we see that $(0)$ is Poisson locally closed, as desired.
\end{proof}

\medskip
\section{Arbitrary base fields of characteristic zero}
\noindent
So far we have restricted our attention to $\mathbb C$-algebras.
It is natural to ask whether our results, both positive and negative, extend to arbitrary base fields.
In this section we will show that everything except the fact that rationality implies primitivity, namely Theorem~\ref{thm: ratprim}, more or less automatically extends to arbitrary characteristic zero fields.

First a word about positive characteristic.
Note that if $A$ is a finitely generated commutative Poisson algebra over a field of characteristic $p>0$, then $a^p$ is in the Poisson centre for every $a\in A$, and in particular it can be shown that for a prime Poisson ideal $P$ of $A$, the notions of Poisson primitive, Poisson rational, and Poisson locally closed are all equivalent to the algebra $A/P$ being a finite extension of the base field. 
Thus we restrict our attention to base fields of characteristic zero.

Let us consider first the construction of Poisson algebras in which $(0)$ is rational but not locally closed.
This was done in Sections~\ref{diffalgsect} and~\ref{sect-poissoncounterexample}.
The only use of the complex numbers in Theorem~\ref{diffex} was that they form an algebraically closed field.
Starting, therefore, with an arbitrary field $k$ of characteristic zero, we obtain, over $L=k^{\alg}$, an affine $L$-algebra $R$ equipped with an $L$-linear derivation $\delta$, such that the field of constants of $\Frac(R)$ is $L$ and the intersection of all nontrivial prime differential ideals of $R$ is zero.
Now, as in the proof of Theorem~\ref{thm: main}, we can write $R=R_0\otimes_FL$ where $F$ is a finite extension of $k$ and $R_0$ is a differential affine $F$-subalgebra of $R$ such that $\Frac(R_0)\cap L=F$.
So the constant field of $\Frac(R_0)$ is $L$, and hence algebraic over $k$.
Since the intersection of a prime differential ideal in $R$ with $R_0$ is prime and differential in $R_0$, we get that the intersection of all prime differential ideals of $R_0$ is also trivial.
We can view $R_0$ as an affine $k$-algebra and that changes neither the fact about the constants of $\Frac(R_0)$, nor the fact about the intersection of the prime differential ideals of $R_0$.
Apply Proposition~\ref{lem: pb} to the $k$-algebra $R_0$ to see that $R_0[t]$ can be endowed with a Poisson bracket such that $(0)$ is not locally closed and the Poisson center of $\Frac\big(R_0[t]\big)$ is equal to the constant field of $R_0$ which is algebraic over $k$.
That is, $(0)$ is rational in $R_0[t]$.
We have thus proved the following generalisation of Corollary~\ref{counterexample}:

\begin{theorem}  Let $k$ be a field of characteristic zero and $d\ge 4$ be a natural number.  Then there exists an affine Poisson $k$-algebra of Krull dimension $d$ such that $(0)$ is Poisson rational but not Poisson locally closed.
\label{counterexample-k}
\end{theorem}

Next we consider the positive statements, that is Theorem~\ref{thm: weak}.
First of all, the proof given there that if a Poisson prime ideal $P$ is contained in only finitely many Poisson prime ideals  of height ${\rm ht}(P)+1$ then $P$ is rational, works verbatim over an arbitrary field of characteristic zero.
The proof of the converse, on the other hand, uses both the uncountability and algebraic closedness of $\mathbb C$, because these are used in the proof of Proposition~\ref{prop: space}.
To deal with this, we require the following lemma, which shows that we can extend scalars and assume that the base field is algebraically closed and uncountable.  We note that given a Poisson bracket $\{-,-\}$ on a $k$-algebra $A$, there is a natural extension of $\{-,-\}$ to a Poisson bracket $\{-,-\}_F$ on $B=A\otimes_k F$ where $F$ is a field extension of $k$.  This is done by defining $\{a\otimes \alpha , b\otimes \beta\}= \{a,b\}\otimes \alpha\beta$ for $\alpha,\beta\in F$ and then extending via linearity.  We call the Poisson bracket $\{-,-\}_F$ the \emph{natural extension} of $\{-,-\}$ to $B$.  
 
\begin{lemma} Let $k$ be a field of characteristic zero and let $A$ be an affine $k$-algebra equipped with a Poisson bracket $\{-,-\}$.
Suppose that $k^{\alg}\cap {\rm Frac}(A)=k$ and that $(0)$ is a Poisson rational ideal of $A$.  Then for any algebraically closed field extension $F$ of $k$, the $F$-algebra $B:=A\otimes_k F$ is again a domain with $(0)$ a Poisson rational ideal with respect to the natural extension of $\{-,-\}$ to $B$.\label{lem: AB}
\end{lemma}

\begin{proof} Since $k^{\alg}\cap {\rm Frac}(A)=k$, the $F$-algebra $B:=A\otimes_k F$ is again a domain.
The Poisson bracket on $A$ extends to a Poisson bracket $\{-,-\}_F$ on $B$ and we claim that $(0)$ is a Poisson rational ideal of $B$. 
Toward a contradiction, suppose that there exists $b/c\in {\rm Frac}(B)\setminus F$ that is in the Poisson centre, with $b,c\in B$, and $c$ nonzero.

We first show that we can witness this counterexample with a finite extension of $k$ rather than $F$.
There is a finitely generated $k$-subalgebra $R$ of $F$ such that $b,c\in A\otimes_k R$.
Let $a_1,\ldots ,a_n\in A$ and $r_1,\ldots ,r_n,s_1,\ldots ,s_n\in R$, some of which are possibly zero, be such that
$b=\sum_{i=1}^n a_i \otimes r_i$ and $c=\sum_{i=1}^n a_i \otimes s_i$.    Since $b\not \in Fc$, we have that there exist $i,j\in \{1,\ldots ,n\}$ such that the $2\times 2$ matrix whose first row is $(r_i,r_j)$ and whose second row is $(s_i,s_j)$ has nonzero determinant.  We let $\Delta\in R$ denote this nonzero determinant.  Since the Jacobson radical of $R$ is zero, there is some maximal ideal $I$ of $R$ such that $\Delta c\not \in A\otimes_k I$.   Then we have a surjection $A\otimes_k R \to A\otimes_k L$ where $L=R/I$ is a finite extension of $k$, and since $k^{\alg}\cap {\rm Frac}(A)=k$, $A\otimes_k L$ is a domain with $k^{\alg}\cap {\rm Frac}(A\otimes_k L) = L$.
By construction, $u:=(b+J)(c+J)^{-1}$ is in the Poisson centre of ${\rm Frac}(A\otimes_k L)$ and is not in $L$ since $\Delta\not\in I$.  

Now let $\{s_1,\ldots ,s_m\}\subseteq L$ be a basis for $L$ over $k$, and hence a basis for $A\otimes_k L$ as a finite and free $A$-module.
Then since ${\rm Frac}(A\otimes_k L)={\rm Frac}(A)\otimes_k L$, we have $u = \sum f_i s_i$ with $f_i\in {\rm Frac}(A)$.
As $u\notin L$, there exists some $f_i$ that is not in~$k$.
Now for any $x\in A$ we have
$\displaystyle 0=\{u,x\}=\sum \{f_i,x\} s_i$, and since the $s_i$ form a basis, we see that each $\{f_i,x\}=0$.
So all the $f_i$ are in the Poisson centre of $\Frac(A)$, contradicting the fact that ${\rm Frac}(A)$ has Poisson center~$k$.
\end{proof}

Now suppose that $A$ is an affine $k$-algebra equipped with a Poisson bracket.  Then $k^{\alg}\cap {\rm Frac}(A)$ is an algebraic extension $K$ of $k$.  In particular, we may replace $k$ by $K$ and replace $A$ by the $K$-subalgebra of ${\rm Frac}(A)$ generated by $K$ and $A$ if necessary and the resulting algebra will still have the property that $(0)$ is Poisson rational.  We may now take an uncountable algebraically closed extension $F$ of $k$ and invoke Lemma \ref{lem: AB} to show that the $F$-algebra $B:=A\otimes_k F$ has the property that $(0)$ is Poisson rational.  By Theorem~\ref{thm: weak}, $B$ has only finitely many height one prime ideals that are Poisson prime ideals.
We point out that it follows that $A$ can only have finitely many height one prime Poisson ideals.
Indeed, let $\{P_1,\ldots ,P_s\}$ be the set of height one prime ideals of $B$ that are Poisson.
By the ``going-down'' property for flat extensions, we see that $Q_i:=P_i\cap A$ must have height at most one in $A$.
So it suffices to show that every height one prime Poisson ideal $Q$ of $A$ is contained in some $P_i$; it will then have to be one of the nonzero $Q_i$ that occurs on this list. 
If $Q$ is a height one prime Poisson ideal of $A$ then the fact that $B$ is a free $A$-module gives that $(A/Q)\otimes_k F$ embeds in $B/QB$.  In particular, by Noether normalization, $B/QB$ contains a polynomial ring over $F$ in $d={\rm Kdim}(B)-1$ variables, where ${\rm Kdim}(B)$ denotes the Krull dimension of $B$, and hence has Krull dimension exactly ${\rm Kdim}(B)-1$.  Since $QB$ is a Poisson ideal, there is a height one prime ideal $Q'$ in $B$ that contains $QB$, which is necessarily a Poisson prime ideal by Lemma \ref{lem: decomp}.  Thus every height one prime Poisson ideal of $A$ is contained in some height one prime Poisson ideal of $B$, as desired.

We have thus proved:

\begin{theorem}
Let $k$ be a field of characteristic zero and $A$ an affine Poisson $k$-algebra.
Then a Poisson prime ideal $P$ of $A$ is rational if and only if the set of Poisson prime ideals $Q\supseteq P$ with ${\rm ht}(Q)={\rm ht}(P)+1$ is finite.
\label{thm: weak-k}
\end{theorem}

There is only remaining the issue of rationality implying primitivity (Theorem~\ref{thm: ratprim}).
Our proof here again uses, in an essential way, that $\mathbb C$ is uncountable.  We note, however, that the proof works in general for any uncountable field (see the remarks following the proof of Theorem~\ref{thm: ratprim}).
We are left therefore with the following open question:

\begin{question}
Suppose $k$ is a countable field of characteristic zero and $A$ an affine Poisson $k$-algebra.
Does rationality of a prime Poisson ideal $P$ imply that $P$ is primitive?
\end{question}

\section{The classical Dixmier-Moeglin equivalence}
\noindent
The counterexamples produced in this paper yield also counterexamples to the classical (noncommutative) Dixmier-Moeglin equivalence discussed in the introduction.  To explain this connection, we recall that given an associative ring $R$ equipped with a derivation $\delta$, one can form an associative \emph{skew polynomial ring} $R[x;\delta]$, which is an overring of $R$ with the property that it is a free left $R$-module with basis $\{x^n : n\ge 0\}$ and such that $xr=rx+\delta(r)$ for all $r\in R$.
Many ring theoretic properties of $R$ are inherited by $R[x;\delta]$; for example,  if $R$ is a domain then so is $R[x;\delta]$ (see \cite[Theorem 1.2.9 (i)]{McConnellRobson}) and if $R$ is left or right noetherian then so is $R[x;\delta]$ (see \cite[Theorem 1.2.9 (iv)]{McConnellRobson}).
Although this skew polynomial construction can be done for any associative ring $R$, we restrict our attention to case when $R$ is commutative.  The ideal structure of $R[x;\delta]$ is intimately connected to the structure of $\delta$-ideals in $R$; indeed, if $I$ is a $\delta$-ideal of $R$ then 
$IR[x;\delta]$ is easily checked to be a two-sided ideal of $R[x;\delta]$.
Using basic facts such as these, as well as some known results about $R[x;\delta]$, we show in Theorem \ref{dmex} below that if $(R,\delta)$ is as in Theorem \ref{diffex} then the skew polynomial ring $R[x;\delta]$ does not satisfy the Dixmier-Moeglin equivalence.  

One interesting feature of the ring $R[x;\delta]$ is that it has finite \emph{Gelfand-Kirillov dimension} whenever $R$ is a finitely generated commutative algebra over a field $k$. We recall that Gelfand-Kirillov dimension (GK-dimension, for short) is a noncommutative analogue of Krull dimension, which is defined as follows. Given a field $k$ and a finitely generated $k$-algebra $A$, a $k$-vector subspace $V\subseteq A$ is called a {\em generating subspace} if it is finite-dimensional, contains $1$, and generates $A$ as a $k$-algebra.
If this is the case we have
$$V\subseteq V^2\subseteq V^3 \subseteq \cdots\subseteq \bigcup_{n\ge 1} V^n = A$$
where $V^n$ denotes the subspace generated by products $v_1v_2\cdots v_n$ with $v_i\in V$.
The \emph{Gelfand-Kirillov dimension} of $A$ is then defined to be
$${\rm GKdim}(A):=\limsup_{n\to\infty} \frac{ \log({\rm dim}(V^n))}{\log\,n}.$$
This quantity is independent of the choice of generating subspace~\cite[Lemma~1.1]{krauselenagan}.  In practice, algebras often have a generating subspace $V$ for which ${\rm dim}(V^n)\sim Cn^d$ for some positive constant $C$ and some $d\ge 0$; in this case $d$ is the GK-dimension.
For a finitely generated commutative $k$-algebra, the Gelfand-Kirillov dimension and the Krull dimension coincide \cite[Theorem 4.5]{krauselenagan}.  

Noetherian noncommutative algebras that do not satisfy the classical Dixmier-Moeglin equivalence seem to be rare. There are very few examples of such algebras in the literature apart from those of Irving and Lorenz mentioned in the introduction, and these are of infinite GK-dimension.  To the best of our knowledge, the following result gives the first counterexamples in finite GK-dimension.

\begin{theorem}
\label{dmex}
With $(R,\delta)$ as in Theorem \ref{diffex}, the skew polynomial ring $R[x;\delta]$ is a noetherian ring of finite $GK$-dimension for which the Dixmier-Moeglin equivalence does not hold.
In particular, $(0)$ is a primitive (and hence rational) prime ideal of $R[x;\delta]$ that is not locally closed in the Zariski topology.
Moreover, for any natural number $n\ge 4$ there exists an example with GK-dimension $n$.
\end{theorem}

\begin{proof}
What Theorem \ref{diffex} gives us is a complex affine algebra $R$ equipped with a derivation $\delta$ such that the field of constants of $\big(\Frac(R),\delta\big)$ is $\mathbb C$, and
the intersection of all nontrivial prime $\delta$-ideals of $R$ is zero.
Given a nonzero prime $\delta$-ideal $P$, we have that 
$Q:=PR[x;\delta]$ is a two-sided ideal of $R[x;\delta]$.
The canonical morphism induces an isomorphism $R[x;\delta]/Q\cong (R/P)[x;\delta']$, where $\delta'$ is the derivation on $R/P$ induced by $\delta$.
Since $R/P$ is an integral domain, so is $(R/P)[x;\delta']$, and hence $Q$ is a nonzero prime ideal of $R[x;\delta]$.
Now, if $a$ is in the intersection of all $Q$'s obtained in this manner, then as $a$ can be written uniquely as $r_nx^n+\cdots+r_0$ for some $n\geq 0$ and $r_0,\dots, r_n\in R$, one sees that the all the $r_i$ must be contained in the intersection of all nontrivial prime $\delta$-ideals of $R$, which we know to be trivial.
It follows that the intersection of all nontrivial prime ideals of $R[x;\delta]$ is trivial, and hence $(0)$ is not locally closed in ${\rm Spec}(R[x;\delta])$.  

The fact that that the field of constants of $\big(\Frac(R),\delta\big)$ is $\mathbb C$ implies that in the commutative algebra $R[z]$ with Poisson bracket given by $\{r,s\}=0$ for $r,s\in R$ and $\{r,z\}=\delta(r)$, the prime ideal $(0)$ is Poisson rational, and hence Poisson primitive by Theorem~\ref{rationalimpliesprimitive}.
By a result of Jordan \cite[Theorem 4.2]{Jordan} it follows that $(0)$ is {\em $\delta$-primitive} in $R$; that is, there is some maximal ideal of $R$ that does not contain a nonzero $\delta$-ideal of $R$.
A result due to Goodearl-Warfield \cite[Corollary 3.2]{GoodWar} now gives that $(0)$ is primitive in $R[x;\delta]$.

Finally, if $R$ is of Krull dimension $m$ then the GK-dimension of $R$ is also $m$ \cite[Theorem~4.5]{krauselenagan}.
Hence the GK-dimension of $R[x;\delta]$ is $m+1$ (see \cite[Proposition~3.5]{krauselenagan}).
So $R[x;\delta]$ is indeed a noetherian and finite GK-dimensional counterexample to the Dixmier-Moeglin equivalence.
Since Theorem~\ref{diffex} gives us such an $R$ of Krull dimension $m$ for any $m\geq 3$, 
 we obtain an example with any integer GK-dimension greater than or equal to  $4$, as claimed.
\end{proof}

 It would be interesting to obtain additional counterexamples.
 More precisely, noting that the Poisson algebra $R[t]$ of Corollary \ref{counterexample} is the semiclassical limit of $R[x;\delta]$ (in the filtered/graded sense \cite[2.4]{Goodearl2}), it is natural to ask:

\begin{question} Do the Poisson algebras of Corollary~\ref{counterexample} admit other formal or algebraic deformations which do not satisfy the classical Dixmier-Moeglin equivalence?
\end{question}


\def\cprime{$'$}

\end{document}